\documentclass[11pt]{article}
\usepackage{amsmath,amstext,amssymb,amsfonts,amsthm}
\usepackage{mathrsfs}
\usepackage{subfigure}
\usepackage{epsfig}
\usepackage[ruled]{algorithm2e}

\usepackage{graphicx}
\usepackage{accents}
\usepackage{color}
\usepackage{bm}
\usepackage[export]{adjustbox}
\usepackage{epstopdf}

\newtheorem{lem}{Lemma}[section]
\newtheorem{thm}{Theorem}[section]

\newtheorem{Def}{Definiton}[section]
\newtheorem{exmp}{Example}

\numberwithin{equation}{section}
\numberwithin{figure}{section}

\newcommand{\be}{\begin{eqnarray}}
\newcommand{\ee}{\end{eqnarray}}
\newcommand{\ben}{\begin{eqnarray*}}
\newcommand{\een}{\end{eqnarray*}}
\newcommand{\beq}{\begin{equation}}
\newcommand{\eeq}{\end{equation}}

\newcommand{\jpl}{[\![}
\newcommand{\jpr}{]\!]}

\newcommand{\calt}{\mathcal{T}}

\newcommand{\lam}{\lambda}

\newcommand{\Om}{\Omega}
\newcommand{\cM}{\mathcal{M}}
\newcommand{\cT}{\mathcal{T}}
\newcommand{\Ga}{\Gamma}
\newcommand{\pa}{\partial}
\newcommand{\de}{\delta}
\newcommand{\supp}{{\rm supp}}
\newcommand{\cE}{\mathcal{E}}

\newcommand{\ls}{[\![}
\newcommand{\rs}{]\!]}
\newcommand{\E}{\mathcal{E}}
\newcommand{\cam}{\mathcal{M}}
\newcommand{\GaK}{\Gamma_{\!\!K}}
\newcommand{\la}{\langle}
\newcommand{\ra}{\rangle}
\newcommand{\R}{\mathbb{R}}
\newcommand{\na}{\nabla}
\newcommand{\al}{\alpha}
\newcommand{\lj}{[{\hskip -1.5pt} [}
\newcommand{\rj}{]{\hskip -1.5pt} ]}

\newcommand{\cS}{\mathcal{S}}
\newcommand{\cP}{\mathcal{P}}

\newcommand{\bbX}{\mathbb{X}}
\newcommand{\bs}{{\backslash}}

\newcommand{\nn}{\nonumber}

\allowdisplaybreaks[4]

\title{An arbitrarily high order unfitted finite element method for elliptic interface problems with automatic mesh generation, Part II. Piecewise-smooth interfaces\footnotemark[1]}
\author{
Zhiming Chen\footnotemark[2]
\and
Yong Liu\footnotemark[3]
}

\date{}

\begin{document}
\maketitle

\renewcommand{\thefootnote}{\fnsymbol{footnote}}
\footnotetext[1]{This work is supported in part by Strategic Priority Research Program of the Chinese Academy of Sciences under the Grant No. XDB0640000, China National Key Technologies R\&D Program under the grant 2019YFA0709602, China Natural Science Foundation under the grant 11831016, 12288201, 12201621 and the Youth Innovation Promotion Association (CAS).}
\footnotetext[2]{LSEC, Institute of Computational Mathematics,
Academy of Mathematics and System Sciences and School of Mathematical Science, University of
Chinese Academy of Sciences, Chinese Academy of Sciences,
Beijing 100190, China. E-mail: zmchen@lsec.cc.ac.cn}
\footnotetext[3]{LSEC, Institute of Computational Mathematics, Academy of Mathematics and Systems Science, Chinese Academy of Sciences, Beijing 100190, P.R. China.  E-mail: yongliu@lsec.cc.ac.cn}

\begin{center}
\small
\begin{minipage}{0.9\textwidth}
\textbf{Abstract.}

We consider the reliable implementation of an adaptive high-order unfitted finite element method on Cartesian meshes for solving elliptic interface problems with geometrically curved singularities.  We extend our previous work on the reliable cell merging algorithm for smooth interfaces to automatically generate the induced mesh for piecewise smooth interfaces. An $hp$ a posteriori error estimate is derived for a new unfitted finite element method whose finite element functions are conforming in each subdomain. Numerical examples illustrate the competitive performance of the method.

\medskip

\end{minipage}
\end{center}
\setlength{\parindent}{2em}
\section{Introduction}
\label{sec_intro}

We propose an adaptive high-order unfitted finite element method on Cartesian meshes with hanging nodes for solving elliptic interface problems, which releases the work of body-fitted mesh generation and allows us to design adaptive finite element methods for solving curved geometric singularities. When the geometry of the interface and domain is simple, such as polygons or polyhedrons, it is well-known that the adaptive finite element method based on a posteriori error estimates initiated in the seminar work \cite{Babuska1978} can achieve quasi-optimal computational complexity \cite{Binev, Nochetto, Stevenson}. The widely used newest vertex bisection algorithm \cite{Bansch} for simplicial meshes can guarantee the shape regularity of adaptively refined meshes starting from an initial shape regular mesh. For curved geometric singularities, however, the newest vertex bisection algorithm can no longer guarantee the shape regularity of the refined meshes. The purpose of this paper is to study an adaptive high-order finite element method and its reliable implementation for arbitrarily shaped piecewise $C^2$-smooth interfaces in the framework of unfitted finite element methods.

The unfitted finite element method in the discontinuous Galerkin (DG) framework is first proposed in \cite{Hansbo}, which is defined on a fixed background mesh and uses different finite element functions in different cut cells that are the intersection of the elements of the mesh with the physical domains. The method has attracted considerable interest in the literature and has been extended in several directions, e.g., the cut finite element method \cite{Burman2010,Burman2012,Hansbo}, the hybrid high order method \cite{Burman2018SINUM,BurmanHHO21}, and the aggregated unfitted finite element method \cite{Badia22,Badia2018}. The {\it small cut cell problem}, which is due to the arbitrarily small or anisotropic intersection of the interface with the elements, is solved by the ghost penalty  \cite{Burman2010, Burman2012, Gurken} or by merging small cut cells with surrounding elements  \cite{Johansson,Huang, Badia2018, Chen2023JCP}. We also refer to the extended finite element methods \cite{Belytschko2001}, the immersed boundary/interface method \cite{Mittal2005, LiNM2003, Li2006, Chen2009}, and the ghost fluid method \cite{LiuJCP2000} for other approaches in designing finite element/finite difference methods on fixed meshes.

In \cite{Chen2021NM}, an adaptive high-order unfitted finite element method based on $hp$ a posteriori error estimates is proposed for the elliptic interface problem based on novel $hp$ domain inverse estimates and the concept of the interface deviation to quantify the mesh resolution of the geometry. The method is defined on the induced mesh of an initial uniform Cartesian mesh so that each element in the induced mesh is a large element by using the idea of cell merging. A reliable algorithm to automatically generate the induced mesh is constructed in \cite{Chen2023JCP} for any $C^2$-smooth interfaces, which is based on the concept of admissible chain of interface elements, the classification of five patterns for merging elements, and appropriately ordering in generating macro-elements from the patterns. This algorithm constitutes an important building block of this paper. We refer to \cite{Melenk2001, Melenk2005, Ern2015} for the $hp$ a posteriori error analysis on conforming finite element meshes and the recent work on a posteriori error analysis for immersed finite element methods \cite{HeJSC2019} and the cut finite element method \cite{Burman2022}.

The main purpose of this paper is to develop a reliable algorithm to automatically generate the induced mesh for any piecewise smooth interfaces. We include a new type of proper intersection of an element with the interface, which is inevitable since the interface is locally like a curved sector around each singular point where the interface is not $C^2$ smooth. We design the singular pattern that merges the element including a singular point with surrounding elements such that the singular pattern is a large element and also its outlet elements have a special structure that allows us to use the reliable algorithm in \cite{Chen2023JCP} to merge the elements in the chain of interface elements connecting two singular patterns. The reliability of the algorithm to generate the induced mesh is also established.

The layout of the paper is as follows. In section 2 we extend the concept of the proper intersection and the large element for elements including singular points and derive an $hp$ a posteriori error estimate for a new unfitted finite element method whose finite element functions are conforming in each subdomain. In section 3 we develop the merging algorithm for piecewise $C^2$-smooth interfaces. In section 4 we report several numerical examples to illustrate the competitive performance of our adaptive method. We also include a numerical example to study the influence of the geometric modeling error by using the adaptive method developed in this paper, which confirms the theoretical result in the appendix of the paper in section 5.

\section{The unfitted finite element method}\label{sec_UFEM}

Let  $\Om\subset \mathbb{R}^2$ be a Lipschitz domain with a piecewise $C^2$-smooth boundary $\Sigma$. The domain $\Om$ is divided by a piecewise $C^2$-smooth interface $\Ga$ into two subdomains $\Om_1,\Om_2$. We assume $\Om_1\subset\subset\Om$ is a Lipschitz domain that is completely included in $\Om$. The intersection points of two pieces of $C^2$ curves of $\Ga$ or $\Sigma$ are called the singular points of the interface or boundary.

We consider the following elliptic interface problem
\begin{align}
&-{\rm div}(a\nabla u)=f\ \ \mbox{in }\Omega_1\cup\Omega_2,\label{m1}\\
&[\![u]\!]_{\Gamma}=0, \, [\![a\nabla u \cdot \mathbf{n}]\!]_{\Gamma}=0\ \ \mbox{on } \Gamma,\ \ u=g\ \ \mbox{on }\Sigma,\label{m2}
\end{align}
where $f\in L^2(\Omega)$, $g\in H^{1/2}(\Sigma)$, $\mathbf{n}$ is the unit outer normal to $\Omega_1$, and $[\![v]\!]:=v|_{\Omega_1}-v|_{\Omega_2}$ stands for the jump of a function $v$ across the interface $\Gamma$. We assume that the coefficient $a(x)$ is positive and piecewise constant, namely, $a=a_1\chi_{\Omega_1}+a_2\chi_{\Omega_2}$, $a_1,a_2 >0$, where $\chi_{\Omega_i}$ is the characteristic function of $\Omega_i$, $i=1,2$.

We assume the domain $\Om$ is covered by a Cartesian mesh $\cT$ with possible hanging nodes, which is obtained by local quad refinements of an initial Cartesian mesh $\cT_0$, see Fig.\ref{fig:2.1}. A Cartesian mesh means the elements of the mesh are rectangles whose sides are parallel to the coordinate axes. We assume each element in $\cT$ does not intersect both $\Ga$ and $\Sigma$. The elements intersecting with the interface or boundary are called the interface elements or boundary elements. We also assume that each element can contain at most one singular point of $\Ga$ or $\Sigma$.  An element that contains a singular point is called a singular element. We define the proper intersection of the interface or boundary with an element $K\in\cT$ as follows.

\begin{figure}[ht!]
\centering
\includegraphics[width=0.25\textwidth]{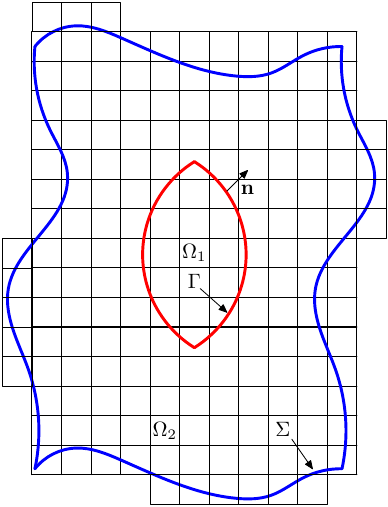}
\hspace{1.0cm}
\includegraphics[width=0.25\textwidth]{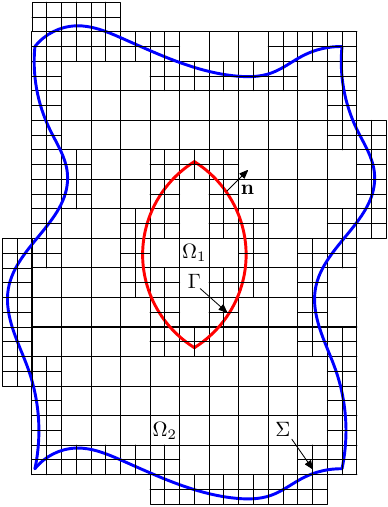}
\caption{Illustration of the domain $\Om$ and the Cartesian mesh $\calt_0$ (left) and $\cT$ (right).}\label{fig:2.1}
\end{figure}

\begin{Def}
\label{def:2.1}
(Proper intersection) The intersection of the interface $\Gamma$ or the boundary $\Sigma$ with an element $K$ is called the proper intersection if $\Gamma$ or $\Sigma$ intersects the boundary of $K$ at most twice at different sides, see Fig.\ref{fig:2.2}(a-b). If $K$ is a singular element, $\Ga$ or $\Sigma$ can also intersect one side of $K$ twice, see Fig.\ref{fig:2.2}(c).
\end{Def}

If $\Gamma$ or $\Sigma$ intersects an element at some vertex $A$, we regard $\Gamma$ as intersecting one of the two edges originated from $A$ at some point very close to $A$. If $\Gamma$ or $\Sigma$ is tangent to an edge of an element, we regard $\Gamma$ or $\Sigma$ as being very close to the edge but not intersecting with the edge.

We call an interface or boundary element $K$ a type $\cT_i$, $i=1,2,3$, element if the interface or boundary intersects the boundary of $K$ at the two neighboring, opposite sides, or the same side, see Fig.\ref{fig:2.2}. We remark that the inclusion of the type $\cT_3$ elements is important when the interface or boundary is piecewise smooth since the interface or boundary is close to a curved sector near the singular points. It is easy to see that when the mesh $\cT$ is locally refined near the interface and boundary, we can always assume that the intersection of $\Ga$ or $\Sigma$ with each element of $\cT$ is the proper intersection.

From $\calt$ we want to construct an induced mesh $\cM$, which avoids possible small intersection of the interface or boundary with the elements of the mesh. We first extend the definition of the large element in Chen et al. \cite[Definition 2.1]{Chen2021NM} to include type $\cT_3$ interface or boundary elements. We denote $\cT^\Ga:=\{K\in\cT:K\cap\Ga\not=\emptyset\}$, $\cT^\Sigma=\{K\in\cT:K\cap\Sigma\not=\emptyset\}$.

\begin{figure}
\centering
\includegraphics[width=0.7\textwidth]{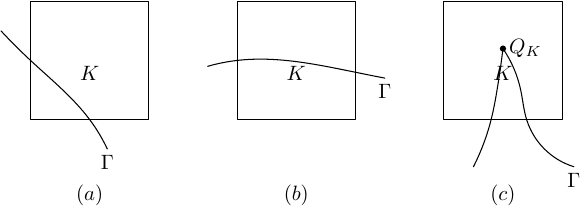}
\caption{Different types of interface elements. From left to right, a type $\cT_1,\cT_2,\cT_3$ element.}\label{fig:2.2}
\end{figure}

\begin{Def}\label{def:2.2}
(Large element) An element $K\in\calt$ is called a large element if $K\subset \Om_i, i=1,2,$ or $K\in \calt^\Gamma\cup\cT^\Sigma$ for which the intersection of $\Gamma$ or $\Sigma$ with $K$ is a proper intersection and there exists a fixed constant $\delta_0\in(0,\frac 12)$ such that $\de_K,\widetilde\de_K\ge\de_0$, where $\de_K=\de_K(\cT)$ is the element geometric index and $\widetilde\de_K=\widetilde\de_K(\cT)$ is the element singular index when $K$ includes a singular point $Q_K$, which are defined by
\ben
\de_K=\min_{\stackrel{i=1,2}{e\in\cE^{\rm side}_K,\,e\cap\Om_i\not=\emptyset}}\frac{|e\cap\Om_i|}{|e|},\ \
\widetilde\de_K=\min_{e\in\cE_K^{\rm side}}\frac{{\rm dist}(Q_K,e)}{(|e^\perp|/2)}.
\een
Here $\cE_K^{\rm side}$ is the set of sides of $K$, ${\rm dist}(Q_K,e)$ is the distance of $Q_K$ to $e\in\cE_K^{\rm side}$, and $e^\perp$ is either one of two sides of $K$ perpendicular to $e\in\cE_K^{\rm side}$.
\end{Def}

The large elements including a singular point of $\Gamma$ or $\Sigma$ will be called {\em singular} large elements. The other kinds of large elements will be called {\em regular} large elements. If the element $K\in \calt^\Gamma\cup\cT^\Sigma $ is not a large element, we make the following assumption as in \cite{Chen2023JCP}, which is inspired by Johansson and Larson \cite{Johansson}.

\medskip
\noindent {\bf Assumption (H1)}: For each $K\in\cT^\Gamma\cup\cT^\Sigma$, there exists a rectangular macro-element $M(K)$ that is a union of $K$ and its surrounding element (or elements) such that $M(K)$ is a large element. We assume $h_{M(K)}\le C_0h_K$ for some fixed constant {{$C_0>0$}}.
\medskip

This assumption can always be satisfied by using the idea of cell merging. In Chen and Liu \cite{Chen2023JCP}, a reliable algorithm to satisfy this assumption is constructed when the interface is $C^2$ smooth. For the piecewise smooth interfaces or boundaries, we will construct the merging algorithm in section \ref{sec_merging_alg}.

In the following, we will always set $M(K)=K$ if $K\in \cT^\Gamma\cup\cT^\Sigma$ is a large element. Then, the induced mesh of $\cT$ is defined as
\ben
\cM=\{M(K):K\in \mathcal{T}^\Gamma\cup\cT^\Sigma\}\cup\{K\in\mathcal{T}: K\not \subset M(K') \text{ for some } K'\in \mathcal{T}^\Gamma\cup\cT^\Sigma\}.
\een
We will write $\cM={\rm Induced}(\cT)$. Note that $\mathcal{M}$ is also a Cartesian mesh in the sense that either $M(K)\cap M(K')=\emptyset$ or $M(K)=M(K')$ for any two different elements $K,K'\in\mathcal{T}$. All elements in $\mathcal{M}$ are large elements.

For any $K\in \cM$, let $h_K$ stand for the diameter of $K$. For $K\in\mathcal{M}^\Gamma:=\{K\in\mathcal{M}:K\cap\Gamma\not=\emptyset\}$, denote $K_i=K\cap\Om_i$, $i=1,2$, $\Ga_K=\Ga\cap K$, and let $\Ga$ intersect the sides of $K$ at $A_K,B_K$. If $\Ga_K$ is smooth, we set $\Ga_K^h=A_KB_K$,
the open line segment connecting $A_K,B_K$. If $K$ includes a singular point $Q_K$, then $\Gamma_K$ is the union of two $C^2$-smooth curves $\Gamma_{1K}\cup \Gamma_{2K}$ where the end points of $\Ga_{1K}$ and $\Ga_{2K}$ are $A_K,Q_K$ and $B_K,Q_K$, respectively. We denote $\Ga_{1K}^h=A_KQ_K, \Ga^h_{2K}=Q_KB_K$, and $\Ga^h_K=\Ga^h_{1K}\cup Q_K\cup\Ga^h_{2K}$. In either case, $\Ga^h_K$ divides $K$ into two polygons $K_1^h, K_2^h$. As the consequence of that $K$ is a large element, we have the following lemma, which can be easily proved and we omit the details.

\begin{figure}[ht!]
\centering
\includegraphics[width=0.8\textwidth]{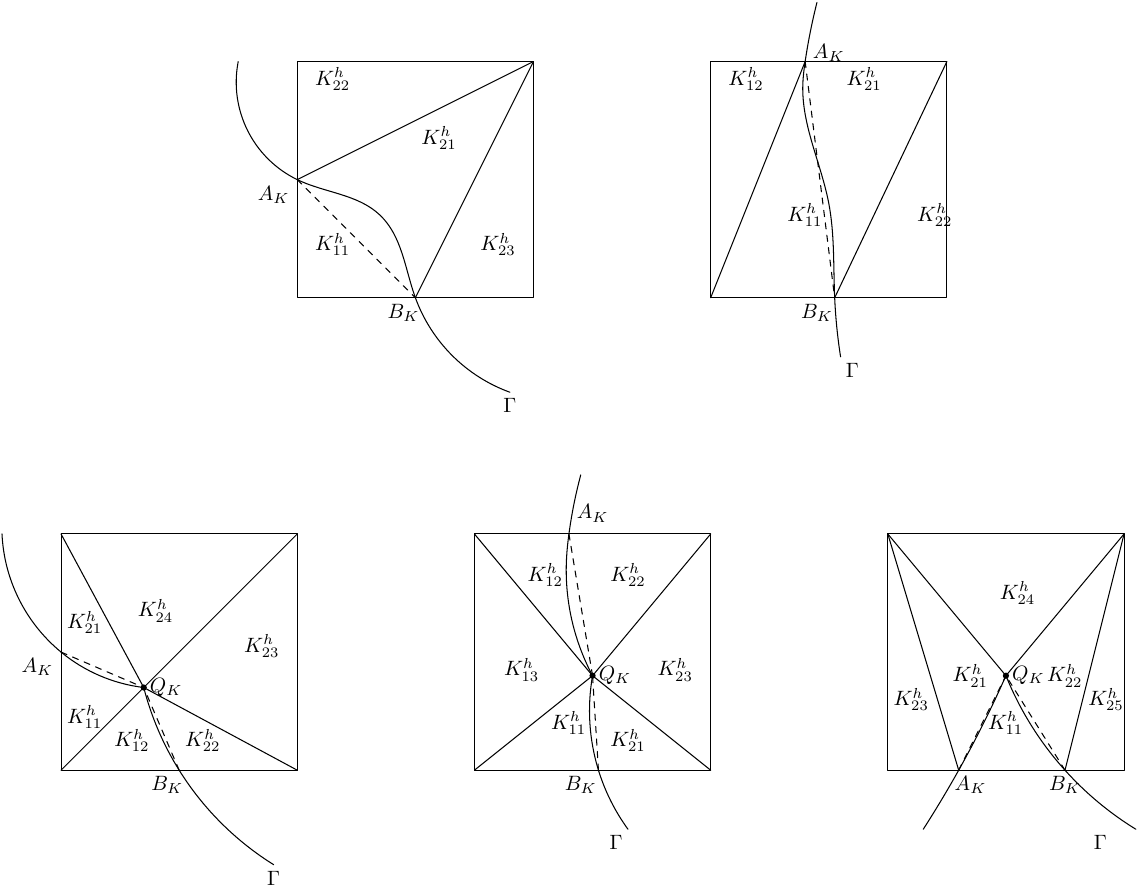}
\caption{Illustration of $K_{ij}^h$ of the {\em regular} interface large elements (top) or {\em singular} interface large elements. $\Gamma_{iK}^h$ are denoted by the dashed line, $i=1,2,j=1,\cdots,J_i^K$. } \label{fig:2.3}
\end{figure}

\begin{lem}\label{lem:2.1}
Let $K\in\cM^\Ga$. Then for $i=1,2$, $K^h_i$ is the union of triangles $K^h_{ij}$, $j=1,\cdots, J_i^K$, $1\le J_i^K\le 5$, such that if $K$ is a regular large element, $K_{ij}^h$ has one vertex at $A_K^i$ that is included in $\Om_i$ and has the maximum distance to $\Ga_K^h$, and if $K$ is a singular large element, $K_{ij}^h$ has one vertex at $Q_K$. The other two vertices of $K_{ij}^h$, $j=1,\cdots,J_i^K$, $i=1,2$, are $A_K,B_K$ or the vertices of $K$ in $\Om_i$, see Fig.\ref{fig:2.3}. Moreover, $K_{ij}^h$, $j=1,\cdots,J_i^K$, $i=1,2$, are shape regular in the sense that the radius of the maximal inscribed circle of $K_{ij}^h$ is bounded below by $c_0h_K$ for some constant $c_0>0$ depending only on $\de_0$ in Definition \ref{def:2.2}.
\end{lem}

If $K_{ij}^h$ has one side as $\Ga_K^h$ or $\Ga_{jK}^h$, $j=1,2$, we define ${\widetilde K}^h_{ij}$ as the curved triangle bounded by the other two sides of $K_{ij}^h$ and one curved side $\Ga_K$ or $\Ga_{jK}^h$, $j=1,2$. If $K_{ij}^h$ does not have a side as $\Ga_K^h$ or $\Ga_{jK}^h$, $j=1,2$, we define ${\widetilde K}^h_{ij}=K_{ij}^h\cap\bar{K}_i$. Then we have
\ben
K=K_1\cup\Ga_K\cup K_2,\ \ K_i=\mbox{\rm the interior of }\overline{\cup^{J_i^K}_{j=1}{\widetilde K}_{ij}^h},\ \ i=1,2.
\een
For $K\in\cM^\Ga$, we set $\cM(K)=\{K_{ij}^h:i=1,2,j=1,\cdots,J_i^K\}, \widetilde{\cM}(K)=\{\widetilde{K}_{ij}^h:i=1,2,j=1,\cdots,J_i^K\}$, and $\sigma_K:\cM(K)\to\widetilde\cM(K)$ the mapping $\sigma_K(K_{ij}^h)=\widetilde{K}_{ij}^h$, $j=1,\cdots,J_i^K$, $i=1,2$,

Similarly, for $K\in\cM^\Sigma:=\{K\in\cM:K\cap\Sigma\not=\emptyset\}$, we denote $\Sigma_K=\Sigma\cap K$ and let $\Sigma$ intersect $K$ at two points $A_K,B_K$. If $\Sigma$ is smooth in $K$, set $\Sigma_K^h=A_KB_K$, and if $K$ is a singular boundary element, denote $\Sigma^h_{1K}=A_KQ_K$, $\Sigma_{2K}^h=B_KQ_K$, and $\Sigma_K^h=\Sigma_{1K}^h\cup Q_K\cup\Sigma_{2K}^h$. We let $\Sigma_{1K}$ the curve of $\Sigma$ with the end points $A_K,Q_K$ and $\Sigma_{2K}$ the curve of $\Sigma$ with the end points $B_K,Q_K$ so that $\Sigma_K=\Sigma_{1K}\cup Q_K\cup\Sigma_{2K}$. $K^h$ is the polygon bounded by $\Sigma^h_K$ and $\pa K\cap\Om$. The proof of the following lemma is also omitted.

\begin{lem}\label{lem:2.2}
Let $K\in\cM^\Sigma$. Then $K^h$ is the union of triangles $K^h_{j}$, $j=1,\cdots, J^K$, $1\le J^K\le 5$, such that if $K$ is a regular large element, $K_{j}^h$ has one vertex at $A_K$ in $\Om$ that has the maximum distance to $\Sigma_K^h$, and if $K$ is a singular large element, $K_{j}^h$ has one vertex at $Q_K$. The other two vertices of $K_{j}^h$, $j=1,\cdots,J_K$, are $A_K,B_K$ or the vertices of $K$ in $\Om$. Moreover, $K_{j}^h$, $j=1,\cdots,J^K$, are shape regular in the sense that the radius of the maximal inscribed circle of $K_{j}^h$ is bounded below by $c_0h_K$ for some constant $c_0>0$ depending only on $\de_0$ in Definition \ref{def:2.2}.
\end{lem}

Again if $K_{j}^h$ has one side as $\Sigma_K^h$ or $\Sigma_{jK}^h$, $j=1,2$, we define ${\widetilde K}^h_{j}$ as the curved triangle bounded by the other two sides of $K_{j}^h$ and one curved side $\Sigma_K$ or $\Sigma_{jK}^h$, $j=1,2$. If $K_{j}^h$ does not have a side as $\Sigma_K^h$ or $\Sigma_{jK}^h$, $j=1,2$, we define ${\widetilde K}^h_{j}=K_{j}^h\cap\bar\Om$. Then we have
\ben
K\cap\Om=\mbox{\rm the interior of }\overline{\cup^{J^K}_{j=1}{\widetilde K}_{j}^h}.
\een
For $K\in\cT^\Sigma$, we denote $\cM(K)=\{K_j^h:j=1,\cdots,J^K\}, \widetilde{\cM}(K)=\{\widetilde{K}^h_j:j=1,\cdots,J^K\}$, and $\sigma_K:\cM(K)\to\widetilde{\cM}(K)$ the mapping $\sigma_K(K_j^h)=\widetilde{K}^h_j$, $j=1,\cdots, J^K$.

For $i=1,2$, let $\cM_i$ be the set of elements $K\in\cM$ included in $\Om_i$, the triangles $T\in{\cT}(K)$, $T\subset K_i^h$ for all elements $K\in\cM^\Ga$, and the triangles $T\subset K^h$ for all elements $K\in\cM^\Sigma$. Then $\cM_i$ is a shape regular body-fitted mixed rectangular and triangular mesh of $\Om_i$, and $\cM_1\cup\cM_2$ is a shape regular body-fitted mixed rectangular and triangular mesh of $\Om$.

The concept of interface deviation introduced in \cite{Chen2021NM} plays an important role in our subsequent analysis. Roughly speaking, the interface deviation measures how far $\Gamma_K$ deviates from $\Gamma_K^h$ or $\Gamma_{jK}$ from $\Gamma_{jK}^h$, $j=1,2$. We now extend this concept to include also the type $\cT_3$ interface and boundary elements.

\begin{Def}\label{def:2.3}
For any $K\in \cM^\Ga\cup\cM^\Sigma$, the interface or boundary deviation $\eta_K$ is defined as
\ben
\eta_K=\max_{\widetilde{T}\in\widetilde{\cM}(K)^*}\frac{{\rm dist}_{\rm H}(E_{T}^h,E_{T})}{{\rm dist}(A_{T},E_{T}^h)},
\een
where $\widetilde{\cM}(K)^*$ is the set of elements $\widetilde{T}\in\widetilde{\cM}(K)$ having one curved side $E_T$, $E_T^h$ is the straight segment connecting the end points of $E_T$, and $A_T$ is the vertex of $T$ opposite to $E_T$.
Here ${\rm dist}_{\rm H}(\Gamma_1,\Gamma_2)=\max_{x\in \Gamma_1}(\min_{y\in \Gamma_2}|x-y|)$and ${\rm dist}(A,\Gamma_1)$ is the distance of a point $A$ to the set $\Gamma_1$.
\end{Def}

It is known \cite{Chen2021NM} that $\eta_K\le Ch_K$. The following assumption, which can be viewed as a variant of the conditions of the mesh resolving the geometry of the interface or boundary, can be easily satisfied if the mesh is locally refined near the interface or boundary.

\medskip
\noindent {\bf Assumption (H2)}: For any $K\in \cM^\Ga\cup\cM^\Sigma$, $\eta_K \leq \eta_0$ for some $\eta_0\in (0,1/2)$.
\medskip

The study in \cite{Chen2023JCP} suggests that one should choose $\eta_0\le 0.1/[p(p+1)]$ in practice to control the condition number of the stiffness matrix of the unfitted finite element method.

Since the mesh $\cT$ may have hanging nodes, we introduce an important concept of $K$-mesh introduced in Babu\v{s}ka and Miller \cite{Babu87} for the Cartesian mesh $\cT$. Let $\mathcal{N}^0$ be the set of conforming nodes of $\cT$, which are the vertices of the elements either located on the boundary $\Sigma$ or shared by the four elements to which they belong. For each conforming node $P$, we denote by $\psi_P\in H^1(\Om^h)$, where $\Om^h=(\cup_{K\in\cT}\bar K)^\circ$, the element-wise bilinear function that satisfies $\psi_P(Q)=\delta_{PQ}\ \forall Q\in\mathcal{N}^0$. Here $\de_{PQ}$ is the Kronecker delta. We impose the following $K$-mesh condition on the mesh $\cT$ and refer to Bonito and Nochetto \cite[\S 6]{Bonito} for a refinement algorithm to enforce the assumption.

\medskip\noindent
{\bf{Assumption (H3)}}: There exists a constant $C>0$ independent of $h_K$ for all $K\in\cT$ such that for any conforming node $P\in\mathcal{N}^0$, ${\rm diam}(\supp(\psi_P))\le C\min_{K\in\cT_P}h_K$, where $\cT_P=\{K\in\cT:K\subset\supp(\psi_P)\}$.
\medskip

For any $p\ge 1$ and any Lipschitz domain $D\subset\R^d$, $d\ge 1$, we denote by $P_p(D)$ the space of polynomials of degree at most $p$ in $D$ and $Q_{p}(D)$ the space of polynomials of degree at most $p$ for each variable in $D$. The unfitted finite element space in Hansbo and Hansbo \cite{Hansbo} and also used in \cite{Chen2021NM}, \cite{Chen2023JCP} is
\ben
\mathbb{H}_p(\cM)=\{v\in L^2(\Om):v=v_1\chi_{\Om_1}+v_2\chi_{\Om_2},v_1|_K,v_2|_K\in Q_p(K), K\in\cM\}.
\een
The finite element functions in this space in general cannot be conforming in each subdomain $\Om_1,\Om_2$. In Chen et al. \cite[Lemma 2.3]{Chen2023waveMC}, it is shown that the mesh $\cM$ constructed by \cite[Algorithm 6]{Chen2023JCP} always satisfies the following compatibility assumption in each subdomain. This allows us to propose a new unfitted finite element space that is conforming in each domain $\Om_1,\Om_2$.

\medskip
\noindent{\bf Assumption (H4):} If $e=\pa G\cap\pa G'$, $G,G'\in\cM_i$, and $F,F'$ be respectively the side of $G,G'$ including $e$, then either $F\subset F'$ or $F'\subset F$.
\medskip

We introduce the following finite element spaces on the interface or boundary elements
\ben
& &X_p(K)=\{v\in H^1(K_1\cup K_2):v|_{\widetilde T}\in P_p(\widetilde{T}), \widetilde{T}\in\widetilde\cM(K)\}\ \ \forall K\in\cM^\Gamma,\\
& &X_p(K)=\{v\in H^1(K\cap\Om):v|_{\widetilde T}\in P_p(\widetilde{T}),\widetilde{T}\in\widetilde\cM(K)\}\ \ \forall K\in\cM^\Sigma.
\een
We also introduce the finite element spaces
\ben
& &U_p(K)=\{v\in H^1(K_1^h\cup K_2^h):v|_T\in P_p(T), T\in\cM(K)\}\ \ \forall K\in\cM^\Ga,\\
& &U_p(K)=\{v\in H^1(K^h):v|_T\in P_p(T),T\in\cM(K)\}\ \ \forall K\in\cM^\Sigma.
\een
The degrees of freedom of the finite elements functions in $X_p(K)$, $K\in\cM^\Ga\cup\cM^\Sigma$, can be defined by using the degrees of freedom of the finite elements functions in $U_p(K)$ via the following natural extension $\mathbb{E}_K:U_p(K)\to X_p(K)$ such that for any $\phi\in U_p(K)$,
\beq\label{z1}
\mathbb{E}_K(\phi)|_{T\cap\widetilde{T}}=\phi|_{T\cap\widetilde{T}}\ \ \forall T\in\cM(K),\ \ \widetilde{T}=\sigma_K(T)\in\widetilde{\cM}(K).
\eeq

Thanks to the compatibility assumption of the mesh $\cM_i$ in the subdomain $\Om_i$, $i=1,2$,
we can define the following unfitted finite element space, which is conforming in
each subdomain $\Om_1,\Om_2$,
\ben
& &\bbX_{p}(\mathcal{M}):=\{v\in H^1(\Om_1\cup\Om_2):  v|_K\in X_{p}(K)\ \ \forall K \in \cM^\Ga\cup\cM^\Sigma,\\
& &\hskip5.2cm v|_K\in Q_{p}(K)\ \ \forall K \in \cM, K\subset\Om_i, i=1,2\}.
\een

In \cite[Lemma 2.4]{Chen2021NM}, an $hp$ domain inverse estimate is proved for unfitted finite element functions in $\mathbb{H}_p(\cM)$. The following domain
inverse estimates for $\mathbb{X}_p(\cM)$ finite element functions are shown in \cite[Lemma 2.5]{Chen2023waveMC} by modifying the argument in \cite[Lemma 2.4]{Chen2021NM}.

\begin{lem}\label{lem:2.3}
Let $K \in \cM^\Ga\cup\cM^\Sigma$. Then there exists a constant $C$ independent of $p$, $h_K$ and $\eta_K$ for all $K\in\cM$ such that for $i=1,2$,
\ben
\|\na v \|_{L^2(K_i)}\leq C p^2 h_K^{-1}\Theta_K^{1/2}\|v\|_{L^2(K_i)}, \quad \|v\|_{L^2(\pa K_i)}\leq C p h_K^{-1/2}\Theta_{K}^{1/2}\|v\|_{L^2(K_i)} \quad \forall v \in X_p(K),
\een
where
\ben
\Theta_K=T\left(\frac{1+3\eta_{K}}{1-\eta_K}\right)^{2p+3}\ \ \forall K\in\cM^\Ga\cup\cM^\Sigma.
\een
Here $\mathsf{T}(t)=t+\sqrt{t^2-1}\ \ \forall t\ge 1$.
\end{lem}

Now we turn to the introduction of the unfitted finite element method. Let $\E=\E^{\rm side}\cup\E^{\rm bdy}$, where $\E^{\rm side}:=\{e=\pa K\cap\pa K':K,K'\in\cM\}\cup\{e=\pa\widetilde{T}\cap\pa\widetilde{T}':\widetilde{T},\widetilde{T}'\in\widetilde{\cM}(K),K\in\cM^\Ga\cup\cM^\Sigma\}$ and $\E^{\rm bdy}:=\{\Sigma_K=\Sigma\cap K: K\in \cam \}$. Set $\E^\Gamma:=\{\GaK=\Gamma \cap K: K\in \cam \}$. Notice that $\cE^\Ga\subset\cE^{\rm side}$ but $\cE^{\rm bdy}\not\subset\cE^{\rm side}$. For any subset $\widehat{\cam}\subset\cam$ and $\widehat{\E}\subset\E$, we use the notation
\begin{align*}
(u,v)_{\widehat{\cam}}=\sum_{K \in \widehat{\cam}}(u,v)_{K}, \ \ \langle u,v\rangle_{\widehat{\E}}=\sum_{e \in \widehat{\E}}\langle u,v\rangle_{e},
\end{align*}
where $(\cdot,\cdot)_K$ and $\la\cdot,\cdot\ra_e$ denote the inner product of $L^2(K)$ and $L^2(e)$, respectively.

For any $e\in \E$, we fix a unit normal vector $\mathbf{n}_e$ of $e$ with the convention that $\mathbf{n}_e$ is the unit outer normal to $\pa\Om$ if $e\in\E^{\rm bdy}$, and $\mathbf{n}_e$ is the unit outer normal to $\pa\Om_1$ if $e\in\cE^\Ga$. Define the normal function $\mathbf{n}|_e=\mathbf{n}_e\ \forall e\in \E$. For any $v\in H^1(\cM):=\{v\in L^2(\Om):v|_{K}\in H^1(K), K\in\cM\}$, we define the jump operator of $v$ across $e$:
\begin{align}
[\![v]\!]|_e:=v^{-} -v^{+}\ \  \forall e \in \E^{\rm side},\ \ \ \
[\![v]\!]|_e:=v^{-}\ \ \forall e \in \E^{\rm bdy},\label{z2}
\end{align}
where $v^{\pm}(\mathbf{x}):=\lim_{\varepsilon\rightarrow 0^+} v(\mathbf{x}\pm\varepsilon \mathbf{n}_e)$ for any $\mathbf{x}\in e$. The mesh function $h|_e=(h_K+h_{K'})/2$ if $e=\pa K\cap\pa K, K,K'\in\cM$, and $h|_e=h_K$ if $e=\pa\widetilde T\cap\pa\widetilde T',\widetilde T,\widetilde T'\in\widetilde{\cM}(K),K\in\cM^\Ga\cup\cM^\Sigma$ or $e=K\cap\Sigma \in\cE^{\rm bdy}$.

For any $v\in H^1(\cM),g\in L^2(\Sigma)$, we define the liftings $\mathsf{L}(v)\in [\bbX_{{p}}(\cM)]^2$, $\mathsf{L}_1(g)\in [\bbX_{{p}}(\cM)]^2$ such that
\beq
(w,\mathsf{L}(v))_\cM=\la w^-\cdot \mathbf{n},\ls v\rs\ra_{\cE^{\Ga}\cup\cE^{\rm bdy}},\ \
(w,\mathsf{L}_1(g))_\cM=\la w\cdot \mathbf{n},g\ra_{\cE^{\rm bdy}}\ \ \ \forall w\in [\bbX_{{p}}(\cM)]^2.\label{ll1}
\eeq

Our unfitted finite element method is to find $U\in \bbX_{{p}}(\cM)$ such that
\beq\label{a2}
a_h(U,v)=F_h(v)\ \ \ \ \forall v\in \bbX_{{p}}(\cM),
\eeq
where the bilinear form $a_h: H^1(\cM)\times H^1(\cM)\to \R$, and the functional $F_h:H^1(\cM)\to\R$ are
given by
\begin{align}
a_h(v,w)=&(a(\na_h v-\mathsf{L}(v)),\na_h w-\mathsf{L}(w))_\cM+\la\al\lj v\rj,\lj w\rj\ra_{{\cE}^\Ga\cup\cE^{\rm bdy}}\nn\\
&+\la p^{-2}h\na_T\lj v\rj,\na_T\lj w\rj\ra_{\cE^\Ga\cup \cE^{\rm bdy}},\label{a3}\\
F_h(v)=&(f,v)_\cM-(a\mathsf{L}_1(g),\na_h v-\mathsf{L}(v))_\cM+\la\al g,v\ra_{\cE^{\rm bdy}}\nn \\
&+\la p^{-2}h\na_T g,\nabla_T v\ra_{\cE^{\rm bdy}},\label{a33}
\end{align}
where $\nabla_T$ is the surface gradient on $\Gamma$ or $\Sigma$. The interface penalty function $\al\in L^\infty({\cE^\Ga\cup\cE^{\rm bdy}})$ is
\beq
\al |_e=\al_0 \hat a_e\widehat\Theta_eh_e^{-1}p^2\ \ \forall e\in{\cE^\Ga\cup\cE^{\rm bdy}},\label{a34}
\eeq
where $\al_0>0$ is a fixed constant, $\hat a_e=\max\{a_K:e\cap \bar{K} \neq \emptyset\}$, $\widehat\Theta_e=\max\{\Theta_K:e\cap \bar K\not=\emptyset\}$. Here $a_K=\|a\|_{L^\infty(K)}$. We remark that \eqref{a2} is a variant of local discontinuous Galerkin method in Cockburn and Shu \cite{Cockburn} on the mesh with curved elements. The stabilization term $\la\al\lj v\rj,\lj w\rj\ra_{\cE^\Ga}$ in \eqref{a2} plays the key role in weakly capturing the jump behavior of the finite element solution at the interface. The boundary condition is also weakly appeared in the \eqref{a2} through the penalty form.

By Lemma \ref{lem:2.3}, we deduce easily that
\beq\label{w3}
\|a^{1/2}\mathsf{L}(v)\|_\cM\le C\|\al^{1/2}\lj v\rj\|_{\cE^\Ga\cup\cE^{\rm bdy}},
\eeq
where the constant $C$ is independent of $p$, $h_K$ and $\eta_K$ for all $K\in\cM$, and the coefficient $a$.
The stability of the bilinear form $a_h(\cdot,\cdot)$ can be proved by using \eqref{w3} and the standard argument in e.g., \cite[Theorem 2.1]{Chen2021NM}, from which, together with the $hp$ interpolation operator in \cite[Theorem 2.1]{Chen2023waveMC}, one can derive an $hp$ a priori error estimate as e.g., in \cite[Theorem 2.1]{Chen2023JCP}. Here we do not elaborate on the details.

Our main goal in this section is to derive the a posteriori error estimate for the solution $U$ of the problem \eqref{a2}, which is the basis for designing adaptive finite element methods for resolving the geometric singularities of the interface and boundary. The key ingredient in the a posteriori error analysis is an $hp$-quasi-interpolation for $H^1$ functions in $X_p(\cM)$ that we now define. Denote $\Om_h=(U_{K\in\cM}\overline{K})^\circ$ be the domain covering $\Om$ and $\mathbb{V}_p(\cM)=\Pi_{K\in \cM}Q_p(K)$. In \cite[Lemma 3.1]{Chen2021NM}, an $hp$-quasi-interpolation operator $\Pi_{hp}:H^1(\Om_h)\to\mathbb{V}_p(\cM)\cap H^1(\Om_h)$ on $K$-meshes is constructed by extending the idea in Melenk \cite{Melenk2005} for the $hp$-quasi-interpolation on conforming meshes. It is shown in \cite[Lemma 3.1]{Chen2021NM} that for any $v\in H^1(\Om_h)$, there exists a constant $C$ independent of $p$, $h_K$ for all $K\in\cM$ such that
\be
& &\|D^m(v-\Pi_{hp}v)\|_{L^2(K)}\le C(h_K/p)^{1-m}\|\na v\|_{L^2(\omega(K))},\ \  m=0,1,\label{y1}\\
& &\|v-\Pi_{hp}v\|_{L^2(\pa K)}\le C(h_K/p)^{1/2}\|\na v\|_{L^2(\omega(K))},\label{y2}
\ee
where $\omega(K)=\{K'\in\cM:K'\subset\supp(\psi_P)\ \forall P\in\mathcal{N}^0\ \mbox{such that }\psi_P|_K\not=0\}$. By Assumption (H3), ${\rm diam}(\omega(K))\le Ch_K$.

\begin{lem}\label{lem:2.4}
There exists a quasi-interpolation operator $\pi_{hp}:H^1(\Om)\to \bbX_p(\cM)$ such that for any $v\in H^1(\Om)$ and $K\in\cM$,
\be
& &\|D^m(v-\pi_{hp}v)\|_{L^2(K)}\le C(h_K/p)^{1-m}\|\na v\|_{L^2(\omega(K))},\ \ m=0,1,\label{y3}\\
& &\|v-\pi_{hp}v\|_{L^2(\pa K)}\le C(h_K/p)^{1/2}\|\na v\|_{L^2(\omega(K))}. \label{y3.5}
\ee
Moreover, for any $\widetilde{T}\in\widetilde{\cM}(K)$, $K\in\cM^\Ga\cup\cM^\Sigma$,
\beq
\|v-\pi_{hp}v\|_{L^2(\pa\widetilde{T})}\le C(h_K/p)^{1/2}\|\na v\|_{L^2(\omega(K))}. \label{y4}
\eeq
The constant $C$ is independent of $p$, $h_K$ and $\eta_K$ for all $K\in\cM$.
\end{lem}

\begin{proof} We first recall the following well-known multiplicative trace inequality
\beq\label{yy1}
\|v\|_{L^2(\pa T)}\le Ch_T^{-1/2}\|v\|_{L^2(T)}+C\|v\|_{L^2(T)}^{1/2}\|\na v\|_{L^2(T)}^{1/2}\ \ \forall v\in H^1(T), T\in\cM(K).
\eeq
Next, for any $K\in\cM^\Ga$, we have the following trace inequality on curved domains in Xiao et al. \cite{Wang}
\beq\label{yy2}
\|v\|_{L^2(\Ga_K)}\le C\|v\|_{L^2(K_i)}^{1/2}\|v\|_{H^1(K_i)}^{1/2}+\|v\|_{L^2(\pa K_i\bs\Ga_K)}\ \ \forall v\in H^1(K), i=1,2.
\eeq
A similar inequality for $\Sigma_K,K\in\cM^\Sigma$ also holds. Then \eqref{y4} follows from \eqref{y3}-\eqref{y3.5} by using \eqref{yy1}-\eqref{yy2}.

Now we prove \eqref{y3}-\eqref{y3.5}. Denote by $\widetilde v\in H^1(\R^2)$ the Stein extension of $v\in H^1(\Om)$ such that $\|\widetilde v\|_{H^1(\R^2)}\le C\|v\|_{H^1(\Om)}$. For $p=2k$ or $p=2k+1$, $k\ge 1$, since $Q_k(K)\subset P_{2k}(K)$, by setting $\pi_{hp}v=\Pi_{hk}(\widetilde v|_{\Om_h})$ we obtain \eqref{y3}-\eqref{y3.5} by \eqref{y1}-\eqref{y2}.

It remains the case when $p=1$. For any $v\in H^1(\Om)$, we let $v_h=\Pi_{h1}(\widetilde v|_{\Om_h})\in\mathbb{V}_1(\cM)\cap H^1(\Om_h)$. By \eqref{y1} we have
\beq
\|D^m(v-v_h)\|_{L^2(K)}\le Ch_K^{1-m}\|\na v\|_{L^2(\omega(K))},\ \  m=0,1.\label{y5}
\eeq
Thus if $K$ is included in $\Om_i$, $i=1,2$, we define $\pi_{h1}v=v_h$ in $K$, which satisfies the desired estimates \eqref{y3}-\eqref{y3.5}. For any $K\in\cM^\Gamma\cup\cM^\Sigma$, we notice that $v_h$ is bilinear in $K$ that is not in $X_1(K)$ . Let $v_K=\frac 1{|K|}\int_Kv d\mathbf{x}$. Then the standard scaling argument yields
\beq
\|D^m(v-v_K)\|_{L^2(K)}\le Ch_K^{1-m}\|\na v\|_{L^2(K)},\ \ m=0,1.\label{y6}
\eeq
Now we are going to lift the piecewise linear polynomial $v_h-v_K$ on $\pa K$ to a finite element function in $U_1(K)$. We only consider the case shown in the right of Fig.\ref{fig:2.3}. The other cases are similar. By the classical polynomial lifting result in Babu\v{s}ka et al. \cite[Lemma 7.1]{Babu} and the scaling argument, we know that there exists a $F_{11}\in P_1(K^h_{11})$ such that $F_{11}=v_h-v_K$ on $A_KB_K$ and for $m=0,1$,
\ben
\|D^mF_{11}\|_{L^2(K_{11}^h)}&\le& Ch_K^{1-m}(h_K^{-1/2}\|v_h-v_K\|_{L^2(A_KB_K)}+|v_h-v_K|_{H^{1/2}(A_KB_K)})\\
&\le&Ch_K^{1-m}\|v_h-v_K\|_{H^1(K)}.
\een
This implies by \eqref{y5}-\eqref{y6} that $\|D^mF_{11}\|_{L^2(K_{11}^h)}\le Ch_K^{1-m}\|\na v\|_{L^2(\omega(K))}$, $m=0,1$. Similarly, by the polynomial lifting in \cite[Lemma 7.2]{Babu} for two sides, there exist $F_{23}\in P_1(K_{23}^h)$ and $F_{25}\in P_1(K_{25}^h)$ such that $F_{23}=v_h-v_K$ on $\pa K\cap\pa K_{23}^h$, $F_{25}=v_h-v_K$ on $\pa K\cap\pa K_{25}^h$, and $\|D^mF_{23}\|_{L^2(K_{23}^h)}+\|D^mF_{25}\|_{L^2(K_{25}^h)}\le Ch_K^{1-m}\|\na v\|_{L^2(\omega(K))}$, $m=0,1$. Since $F_{23}=F_{11}$ at $A_K$, $F_{25}=F_{11}$ at $B_K$, we can define linear lifting functions $F_{21}\in P_1(K_{21}^h), F_{22}\in P_1(K_{22}^h)$ and then $F_{24}\in P_1(K_{24}^h)$ such that $F_{24}=v_h-v_K$ on $\pa K\cap K_{24}^h$ and
\ben
\|D^mF_{21}\|_{L^2(K_{21}^h)}+\|D^mF_{22}\|_{L^2(K_{22}^h)}+\|D^m F_{24}\|_{L^2(K_{24}^h)}\le Ch_K^{1-m}\|\na v\|_{L^2(\omega(K))},\ \ m=0,1.
\een
Now define $F\in U_1(K)$ by $F|_{K_{ij}^h}=F_{ij}$, $i=1,2$, $j=1,\cdots, J_i^K$, then we have $\|D^m F\|_{L^2(K)}\le Ch_K^{1-m}\|\na v\|_{L^2(\omega(K))}$. Let $\pi_{h1}v=v_K+\mathbb{E}_K(F)$, where $\mathbb{E}_K(F)$ is defined in \eqref{z1}. Then $\pi_{h1}v\in X_1(K)$, $\pi_{h1}v=v_h$ on $\pa K$, and since $p=1$, by Lemma \ref{lem:2.3} and \eqref{y6} we know that
\ben
\|D^m(v-\pi_{h1}v)\|_{L^2(K)}&\le&\|D^m(v-v_K)\|_{L^2(K)}+\|D^m\mathbb{E}_K(F)\|_{L^2(K)}\\
&\le&\|D^m(v-v_K)\|_{L^2(K)}+C\|D^m F\|_{L^2(K)}\\
&\le&Ch_K^{1-m}\|\na v\|_{L^2(\omega(K))},\ \ m=0,1.
\een
This shows \eqref{y3}. The estimate \eqref{y3.5} follows from \eqref{y3} by the multiplicative trace inequality. This completes the proof.
\end{proof}

Let $U\in \bbX_{{p}}(\cam)$ be the solution of the problem \eqref{a2}, we define the element and jump residuals
\begin{align*}
R(U)|_K=f+{\rm div}_h(a\nabla_h U) \quad \forall K\in \cam,\ \
J(U)|_{e}=\jpl a \nabla_h U\cdot {\bf n}\jpr|_{e}\quad  \forall e \in \cE^{\rm side}.
\end{align*}
We also define the functions $\Lambda : \Pi_{K\in \cam}L^2(K)\rightarrow \mathbb{R}$ and $\hat{\Lambda}: \Pi_{e\in \cE}L^2(e)\rightarrow \mathbb{R}$ as
\begin{align*}
\Lambda|_{K}=\|a^{1/2}\|_{L^\infty(K)}\|a^{-1/2}\|_{L^\infty(\omega(K))} \quad \forall K \in \cam,\ \
\hat{\Lambda}|_{e}=\max \{\Lambda_K: e\cap \bar{K} \neq \emptyset\}\quad \forall e\in \cE.
\end{align*}

\begin{lem}\label{lem:2.5}
Let $w\in H^1_0(\Om)$ and $w_h=\pi_{hp}w\in\mathbb{X}_p(\cM)$ be defined in Lemma \ref{lem:2.4}. Define $\zeta\in L^\infty(\Om_h)$ by $\zeta|_K=p^{1/2}\Theta_K^{1/2}\Lambda_K\ \forall K\in\cM$. Then there exists a constant $C$ independent of $p$, $h_K$ and $\eta_K$ for all $K\in\cM$, and the coefficient $a$ such that $\|a^{1/2}\zeta^{-1}\mathsf{L}(w_h)$$\|_\cM\le C\|a^{1/2}\na w\|_\cM$.
\end{lem}

\begin{proof} By the definition of the lifting operator in \eqref{ll1} we know that $(v,\mathsf{L}(w_h))_\cM=\la v^-\cdot\mathbf{n},\lj w_h\rj\ra_{\cE^\Ga\cup\cE^{\rm bdy}} \ \forall v\in[\mathbb{X}_p(\cM)]^2$. By taking $v=a\zeta^{-2}\mathsf{L}(w_h)\in[\mathbb{X}_p(\cM)]^2$, we have
\beq\label{w1}
\|a^{1/2}\zeta^{-1}\mathsf{L}(w_h)\|_\cM^2=\sum_{K\in\cM^\Ga}\la [a\zeta^{-2}\mathsf{L}(w_h)]^-\cdot\mathbf{n},\lj w_h\rj\ra_{\Ga_K}
+\sum_{K\in\cM^\Sigma}\la [a\zeta^{-2}\mathsf{L}(w_h)]^-\cdot\mathbf{n},\lj w_h\rj\ra_{\Sigma_K}.
\eeq
By Lemma \ref{lem:2.3} and Lemma \ref{lem:2.4}, for any $K\in\cM^\Ga$,
\ben
\la [a\zeta^{-2}\mathsf{L}(w_h)]^-\cdot\mathbf{n},\lj w_h\rj\ra_{\Ga_K}&\le&
Ca_K^{1/2}(\zeta|_K)^{-2}p^{1/2}\Theta^{1/2}_K\|a^{1/2}\mathsf{L}(w_h)\|_{L^2(K)}\|\na w\|_{L^2(\omega(K))}\\
&\le&C\|a^{1/2}\zeta^{-1}\mathsf{L}(w_h)\|_{L^2(K)}\|a^{1/2}\na w\|_{L^2(\omega(K))}.
\een
Similarly, for any $K\in\cM^\Sigma$,
\ben
\la [a\zeta^{-2}\mathsf{L}(w_h)]^-\cdot\mathbf{n},\lj w_h\rj\ra_{\Sigma_K}\le C\|a^{1/2}\zeta^{-1}\mathsf{L}(w_h)\|_{L^2(K)}\|a^{1/2}\na w\|_{L^2(\omega(K))}.
\een
This completes the proof by inserting above two estimates to \eqref{w1}.
\end{proof}

For any $v\in H^1(\cM)$, we define the following DG norm
\ben
\| v \|_{\rm DG}^2= \|a^{1/2}\na v\|_{\cM}^2+\|\alpha^{1/2}\lj v \rj \|_{\cE^\Ga\cup\cE^{\rm bdy}}^2+\|p^{-1}h^{1/2}\nabla_T \lj v \rj\|_{\cE^\Ga\cup \cE^{\rm bdy}}^2.
\een
The following theorem is the main result of this section.

\begin{thm}\label{thm:2.1}
Let $u\in H^1(\Omega)$ be the weak solution of \eqref{m1}-\eqref{m2} with $g\in H^1(\partial \Omega)$ and $U\in \bbX_{{p}}(\cam)$ be the solution of \eqref{a2}. Then there exists a constant $C$ independent of $p$, $h_K$ and $\eta_K$ for all $K\in\cM$, and the coefficient $a$ such that
\begin{align*}
\|u-U\|_{\rm DG}  \leq C\left(\sum_{K\in\cam}\xi_K^2\right)^{1/2},
\end{align*}
where for each $K\in \cam$, the local a posteriori error estimator
\begin{align*}
\xi_K^2=&\left(\|a^{-1/2}(h/p)\Lambda R(U)\|_K^2+\|\hat{a}^{-1/2}(h/p)^{1/2}\hat{\Lambda}J(U)\|_{\cE_K}^2\right)\\
&+\left(\|\al^{1/2}p^{1/2}\widehat{\Theta}^{1/2} \hat{\Lambda} \jpl U \jpr\|_{\Ga_K}^2+\|\al^{1/2}p^{1/2}\widehat{\Theta}^{1/2}\hat\Lambda(U-g)\|_{\Sigma_K}^2\right)\\
&+\left(\|\hat{a}^{1/2} p^{-1}h^{1/2}\widehat{\Theta}^{1/2}\hat\Lambda\nabla_T\jpl U\jpr\|_{\Gamma_K}^2+\|\hat{a}^{1/2} p^{-1}h^{1/2}\widehat{\Theta}^{1/2}\hat\Lambda\nabla_{T}(U-g)\|_{\Sigma_K}^2\right).
\end{align*}
Here $\cE_K=\{e\in\cE^{\rm side}:e\cap\bar K\not=\emptyset\}$.
\end{thm}

\begin{proof} The proof modifies the argument in \cite[Theorem 3.1]{Chen2021NM}, which extends the argument for deriving a posteriori error estimates for DG methods in e.g., Karakashian and Pascal \cite{Pascal} and \cite{Bonito}. Let $\widetilde U\in H^1(\Om)$ such that $\widetilde{U}=g$ on $\pa\Om$, and
\beq\label{ww1}
\int_\Om a\na\widetilde U\cdot\na vdx=\int_{\Om}a\na_h U\cdot\na v dx\ \ \forall v\in H^1_0(\Om).
\eeq
By the Lax-Milgram lemma, $\widetilde U$ is well-defined. By the triangle inequality, since $\lj u-\widetilde U\rj=0$ on $\cE^\Ga\cup\cE^{\rm bdy}$,
\be
\|u-U\|_{\rm DG}&\le&\|u-\widetilde{U}\|_{\rm DG}+\|\widetilde{U}-U\|_{\rm DG}\nn\\
&\le&\|a^{1/2}\na(u-\widetilde{U})\|_\cM+\|a^{1/2}\na_h(U-\widetilde U)\|_\cM+\|\al^{1/2}\lj U\rj\|_{\cE^\Ga}+\|p^{-1}h^{1/2}\na_T\lj U\rj\|_{\cE^{\Ga}}\nn\\
& &\quad+\|\al^{1/2}(U-g)\|_{\cE^{\rm bdy}}+\|p^{-1}h^{1/2}\na_T(U-g)\|_{\cE^{\rm bdy}}.\label{ww2}
\ee
We now estimate the first two terms in \eqref{ww2}. We first estimate the conforming component $\|a^{1/2}\na(u-\widetilde U)\|_\cM$. Set  $w=u-\widetilde U\in H^1_0(\Om)$ and let $w_h=\pi_{hp}w\in\mathbb{X}_p(\cM)$ be defined in Lemma \ref{lem:2.4}. By \eqref{a3} we have
\be
(a\na(u-\widetilde{U}),\na w)_\cM&=&(f,w)_\cM-(a\na_hU,\na_h w)_\cM\nn\\
&=&(f,w-w_h)_\cM-(a\na_h U,\na_h(w-w_h))_\cM\nn\\
& &-(a\na_h U,\mathsf{L}(w_h))_\cM+(a(\mathsf{L}_1(g)-\mathsf{L}(U)),\na_hw_h-\mathsf{L}(w_h))_\cM\nn\\
& &+\la\al\lj U\rj,\lj w_h\rj\ra_{\cE^\Ga}+\la p^{-2}h\na_T\lj U\rj,\na_T\lj w_h\rj\ra_{\cE^\Ga}\nn\\
& &+\la\al(U-g),w_h\ra_{\cE^{\rm bdy}}+\la p^{-2}h\na_T(U-g),\na_T\lj w_h\rj\ra_{\cE^{\rm bdy}}\nn\\
&:=&{\rm I}_1+\cdots+{\rm I}_8.\label{ww3}
\ee
By using integration by parts and the DG magic formula $\lj(a\na_h U\cdot\mathbf{n})(w-w_h)\rj=\lj a\na_hU\cdot\mathbf{n}\rj(w-w_h)^++(a\na_hU\cdot\mathbf{n})^-\lj w-w_h\rj$ on any $e\in\cE$, we have
\ben
{\rm I}_1+{\rm I}_2+{\rm I}_3=(R(U),w-w_h)_\cM-\la J(U),(w-w_h)^+\ra_{\cE^{\rm side}}.
\een
Now by using Lemma \ref{lem:2.4}, we get
\ben
|{\rm I}_1+{\rm I}_2+{\rm I}_3|\le C(\|a^{-1/2}(h/p)\Lambda R(U)\rj_\cM+\|\hat a^{-1/2}(h/p)^{1/2}\hat{\Lambda}J(U)\|_{\cE^{\rm side}})\|a^{1/2}\na w\|_\cM.
\een
Next by the definition of the lifting operators in \eqref{ll1}
\ben
{\rm I}_4=-\la a(\na_hw_h-\mathsf{L}(w_h))^-\cdot\mathbf{n},\lj U\rj\ra_{\cE^\Ga}+\la a(\na_hw_h-\mathsf{L}(w_h))^{-}\cdot\mathbf{n},g-U\ra_{\cE^{\rm bdy}}.
\een
By Lemma \ref{lem:2.3} and Lemma \ref{lem:2.5}, we have
\ben
|{\rm I}_4|\le C(\|\al^{1/2}\zeta\lj U\rj\|_{\cE^\Ga}+\|\al^{1/2}\zeta(U-g)\|_{\cE^{\rm bdy}})\|a^{1/2}\na w\|_\cM.
\een
The other terms can be estimated similarly
\ben
|{\rm I}_5+\cdots+{\rm I}_8|&\le&C(\|\al^{1/2}\zeta\lj U\rj\|_{\cE^\Ga}+\al^{1/2}\zeta(U-g)\|_{\cE^{\rm bdy}})\|a^{1/2}\na w\|_\cM\\
& &+C\|\hat{a}^{1/2} p^{-1}h^{1/2}\widehat{\Theta}^{1/2}\hat\Lambda\nabla_T\jpl U\jpr\|_{\cE^\Gamma}\|a^{1/2}\na w\|_\cM\\
& &+C\|\hat{a}^{1/2} p^{-1}h^{1/2}\widehat{\Theta}^{1/2}\hat\Lambda\nabla_{T}(U-g)\|_{\cE^{\rm bdy}}\|a^{1/2}\na w\|_\cM.
\een
Substituting the estimates for ${\rm I}_1,\cdots, {\rm I}_8$ into \eqref{ww3}, we obatin
\beq\label{ww4}
\|a^{1/2}\na(u-\widetilde{U})\|_\cM\le C\left(\sum_{K\in\cM}\xi_K^2\right)^{1/2}.
\eeq
The nonconforming component $\|a^{1/2}\na_h(U-\widetilde{U})\|_\cM$ can be estimated by the same argument as that in \cite[Theorem 3.1]{Chen2021NM} to obtain
\ben
\|a^{1/2}\na_h(U-\widetilde{U})\|_\cM&\le&C(\|\hat a^{1/2}ph^{-1/2}\lj U\rj\|_{\cE^\Ga}+\|\hat a^{1/2}p^{-1}h\na_T\lj U\rj \|_{\cE^\Ga})\\
& &\quad +C(\|\hat a^{1/2}ph^{-1/2}(U-g)\|_{\cE^{\rm bdy}}+\|\hat a^{1/2}p^{-1}h^{1/2}\na_T(u-g)\|_{\cE^{\rm bdy}}).
\een
This, together with \eqref{ww2} and \eqref{ww4}, completes the proof.
\end{proof}

We remark that in above proof $\mathsf{L}(w_h)\not=0$, which is different from the proof in \cite[Theorem 3.1]{Chen2021NM} where $\mathsf{L}(w_h)=0$ due to the unfitted finite element space being $\mathbb{H}_p(\cM)$ and the domain being assumed to be a union of rectangles so that $\pi_{hp}w$ can be chosen in $\mathbb{H}_p(\cM)\cap H^1_0(\Om)$.

The lower bound of the a posteriori error estimator $\xi_K$ can also be established by the similar argument as that in \cite[Theorem 4.1]{Chen2021NM}. We leave it to the interested readers.

\section{The merging algorithm}\label{sec_merging_alg}

In this section, we propose the merging algorithm to generate the induced mesh for the piecewise smooth interface and boundary. A reliable algorithm to generate the induced mesh is developed in \cite[Algorithm 6]{Chen2023JCP} for arbitrarily shaped $C^2$-smooth interfaces. The algorithm is based on the concept of the admissible chain of interface elements, the classification of patterns for merging elements, and appropriate ordering in generating macro-elements. Our strategy to treat the piecewise smooth interface or boundary is first to design singular patterns that are large macro-elements surrounding each singular point of the interface or boundary, then use  \cite[Algorithm 6]{Chen2023JCP} to deal with the remaining interface or boundary elements inside which the interface or boundary is $C^2$-smooth.

We now recall some notation from \cite{Chen2023JCP}.
A chain of interface or boundary elements $\mathfrak{C}=\{G_1\rightarrow G_2 \rightarrow \cdots \rightarrow G_n\}$ orderly consists of $n$ interface or boundary elements $G_i \in \mathcal{T}^{\Gamma}$ or $G_i\in\cT^\Sigma$, $i=1,\cdots,n$, such that $\bar\Gamma_{G_i}\cup\bar\Gamma_{G_{i+1}}$ is a continuous curve, $1\le i\le n-1$. We call $n$ the length of $\mathfrak{C}$ and denote $\mathfrak{C}\{i\}=G_i$, $i=1,\cdots,n$.

For any $K\in\cT$, we call $N(K)\in\cT$ a neighboring element of $K$ if $K$ and $N(K)$ share a common side. Set $\cS(K)_0=\{K\}$, and for $j\ge 1$, denote $\cS(K)_j=\{K''\in\cT:\exists\,K'\in\cS(K)_{j-1}\ \mbox{such that }\bar K''\cap\bar K'\not=\emptyset\}$, that is, $\cS(K)_j$ is the set of all $k$-th layer
elements surrounding $K$, $0\le k\le j$. Obviously, $\cS(K)_0\subset\cS(K)_1\subset\cdots\subset\cS(K)_j$ for any $ j\ge 1$.

The following definition of the admissible chain is introduced in \cite{Chen2023JCP} for interface elements. We include here also for the boundary elements.

\begin{Def}\label{def:3.1}
A chain of interface or boundary elements $\mathfrak{C}$ is called admissible if the following rules are satisfied.
\begin{description}
\item[$1.$] For any $K\in\mathfrak{C}$, all elements in $\cS(K)_2$ have the same size as
that of $K$.
\item[$2.$] If $K\in\mathfrak{C}$ has a side $e$ such that $\bar e\subset \Om_i$, then $e$ must be a side of some neighboring element $N(K)\subset\Om_i$, $i=1,2$.
\item[$3.$] Any elements $K\in\cT\backslash\cT^\Gamma\cup\cT^\Sigma$ can be neighboring at most two elements in $\mathfrak{C}$.
\item[$4.$] For any $K\subset\Om_i$, the interface or boundary elements in {$\cS(K)_j$, $j=1,2$,} must be connected in the sense that the interior of the closed
set $\cup\{\bar G: G\in\cS(K)_{j}\cap{{\cT}}^\Ga\}$ or $\cup\{\bar G: G\in\cS(K)_{j}\cap{{\cT}}^\Sigma\}$ is a connected domain.
\end{description}
\end{Def}

Notice that for the smooth interface or boundary, only type $\cT_1,\cT_2$ elements occur in the admissible chain, the following theorem is proved in
\cite[Theorem 3.1]{Chen2023JCP} for the admissible chain of interface elements whose start and end elements are of type $\cT_2$ elements. A closer check of the proof reveals that the conclusion also holds when the start or end elements are two neighboring type $\cT_1$ elements. The extension of the theorem to the chain of boundary elements is straightforward. Here we omit the details.

\begin{thm}\label{thm:3.1}
Let $\delta_0\in (0,1/5]$. For any admissible chain of interface or boundary elements $\mathfrak{C}$ with length $n\geq 2$, if the start elements of $\mathfrak{C}$ satisfy $\mathfrak{C}(1)\in \mathcal{T}_2$ or $\mathfrak{C}(1),\mathfrak{C}(2)\in \cT_1$ and the end elements of $\mathfrak{C}$ also satisfy $\mathfrak{C}(n)\in \mathcal{T}_2$ or $\mathfrak{C}(n-1),\mathfrak{C}(n)\in \cT_1$, then the merging algorithm in \cite[Algorithm 6]{Chen2023JCP} can successfully generate a locally induced mesh with input $\mathfrak{C}$.
\end{thm}

Here and in the following, we call a locally induced mesh of a chain $\mathfrak{C}$ if it is a union of large elements that covers the interface or boundary included in the elements in $\mathfrak{C}$.

\subsection{The singular pattern}

We start by introducing the definition of the singular pattern.

\begin{Def}\label{def:3.2}
For each singular element $K\in\cT$, there exists a macro-element $M(K)$ that is a union of elements surrounding $K$ such that $M(K)$ is a large element. Moreover, if $G_1,G_2\in\cT$ are two interface or boundary elements in $\cS(M(K))_1\backslash M(K)$, then for $i=1,2$, either $G_i\in\cT_2$ or $G_i\in\cT_1$ in which case there is a neighboring element $G_i'\in\cT_1$ in $\cS(M(K))_1\backslash M(K)$. Denote $\cP_i=\{G_i\}$ if $G_i\in\cT_2$ or $\cP_i=\{G_i,G_i'\}$ if $G_i\in\cT_1$. We assume ${\rm dist}(\cP_1,\cP_2)\ge 2$, where ${\rm dist}(\cP_1,\cP_2)$ is the minimum number of non-interface or non-boundary elements in $\cS(M(K))_1\backslash M(K)$ connecting $\cP_1,\cP_2$, see Fig.\ref{fig:x1}. We call $M(K)$ the singular pattern including $K$ and $\cP_1$, $\cP_2$ the outlet patches of the macro-element $M(K)$.
\end{Def}

\begin{figure}[!ht]
\centering
{\includegraphics[width=0.8\textwidth]{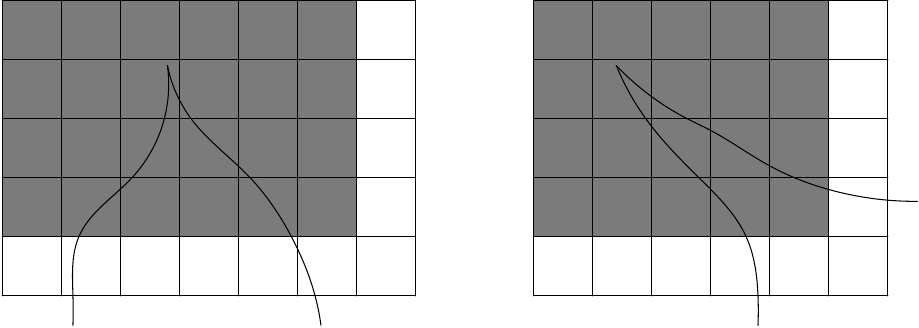}}
\caption{Examples of the singular pattern.} \label{fig:x1}
\end{figure}

We recall that $\cS(M(K))_1\backslash M(K)$ is the union of the first layer elements in $\cT$ surrounding $M(K)$. The assumptions on the patch of elements $\cP_1,\cP_2$ are important for us to use \cite[Algorithm 6]{Chen2023JCP} to construct macro-elements for all remaining elements in $\cT^\Ga$ or $\cT^\Sigma$ inside which the interface or boundary is $C^2$ smooth.

Since for each singular element $K$, $K_i=K\cap\Om_i$, $i=1,2$, is locally a curved sector, the following lemma indicates that one can always construct a singular pattern $M(K)$ if the mesh is locally refined around each singular point.

\begin{lem}\label{lem:3.1}
Let $S$ be a sector bounded by two half lines $L_1,L_2$ with vertex $Q\in K$, $K\in\cT$. Then there exists a rectangular macro-element $M(K)$ that is a large element and for $i=1,2$, $\cP_i=\{K'\in\cT: K'\subset \cS(M(K))_1\backslash M(K), K'\cap L_i\not=\emptyset\}$ includes either one $\cT_2$ element or two neighboring $\cT_1$ elements. Moreover, ${\rm dist}(\cP_1,\cP_2)\ge 2$ and the element singular index $\widetilde\de_{M(K)}$ is independent of the element size $h_K$.
\end{lem}

\begin{figure}[!ht]
\centering
\includegraphics[width=0.8\textwidth]{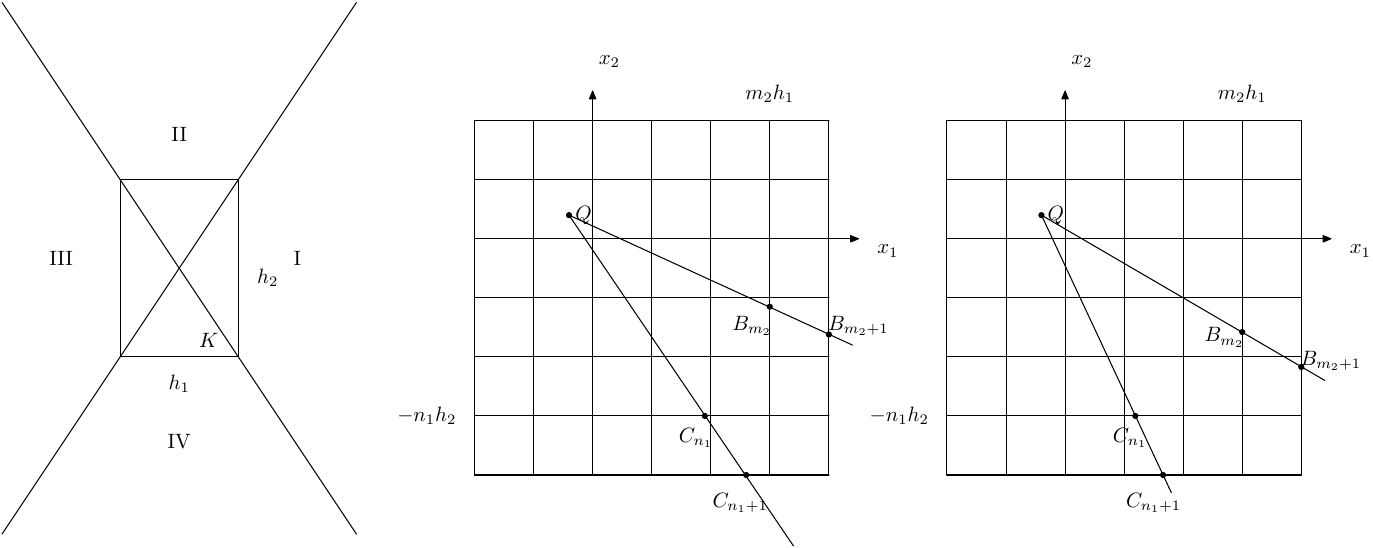}
\caption{The figures used in the proof of Lemma \ref{lem:3.1}.}\label{fig:x2}
\end{figure}

\begin{proof}
Let $h_1,h_2$ be the length of the horizontal and vertical sides of $K$, respectively. Let $Q=(-\al h_1,\beta h_2)$, $\al,\beta\in (0,1)$, be the vertex of the sector $S$, and for $i=1,2$, the equation of $L_i$ is $x_2-\beta h_2=k_i(x_1+\al h_1)$ with the convention that if $k_i=\infty$, the equation is $x_1+\al h_1=0$. It is clear that if $|k_i|< h_2/h_1$, $L_i$ tends to the infinity in the region ${\rm I, III}$ marked in Fig.\ref{fig:x2}(left), and if $|k_i|>h_2/h_1$, $L_i$ tends to the infinity in the region ${\rm II, IV}$ in Fig.\ref{fig:x2}(left). Let $M(K)=(-m_1h_1,m_2 h_1)\times (-n_1h_2,n_2h_2)$, $m_i,n_i\in\mathbb{N}$, $i=1,2$,
be the macro-element.

In the following, we prove the lemma when $|k_1|\le h_2/h_1, |k_2|>h_2/h_1$.
We first assume $k_1,k_2<0$ so that $\lam_1:=k_1h_1/h_2\in (-1,0)$, $\lam_2:=k_2^{-1}h_2/h_1\in (-1,0)$. Let $m_1=2, n_2=2$ in the macro-element $M(K)$, see Fig.\ref{fig:x2}(middle) and Fig.\ref{fig:x2}(right). Denote by $B_{m_2}=(m_2h_1,b_{m_2}), B_{m_2+1}=((m_2+1)h_1,b_{m_2+1})$ the intersection points of $L_1$ with the lines $x_1=m_2h_1, x_1=(m_2+1)h_1$, and $C_{n_1}=(c_{n_1},-n_1h_2), C_{n_1+1}=(c_{n_1+1}, -(n_1+1)h_2)$ the intersection points of $L_2$ with the lines $x_2=-n_1h_2, x_2=-(n_1+1)h_2$. Obviously, $b_{m_2}-b_{m_2+1}=-k_1h_1\le h_2$, which implies that $L_1$ intersects $\cS(M(K))_1\backslash M(K)$ either with one $\cT_2$ element or two neighboring $\cT_1$ elements. Similarly, $L_2$ intersects $\cS(M(K))_1\backslash M(K)$ with either one $\cT_2$ element or two neighboring $\cT_1$ elements.

In order to guarantee that ${\rm dist}(\cP_1,\cP_2)\ge 2$, we require $b_{m_2+1}\ge -(n_1-1)h_2, c_{n_1}<(m_2-1)h_2$ in the case of Fig.\ref{fig:x2}(middle) or $b_{m_2}\ge -(n_1-1)h_2, c_{n_1+1}\le (m_2-1)h_1$ in the case of Fig.\ref{fig:x2}(right). In the case shown in Fig.\ref{fig:x2}(middle), it is easy to see that $b_{m_2+1}\ge -(n_1-1)h_2, c_{n_1}<(m_2-1)h_2$  is equivalent to $-\lam_1(m_2+1+\al)\le n_1+\beta-1$, $-\lam_2(n_1+\beta)<m_2+\al-1$, which implies $n_1+\beta\ge (1-2\lam_1)/(1-\lam_1\lam_2), m_2+\al+1>(2-\lam_2)/(1-\lam_1\lam_2)$. Thus one can choose
\ben
m_2=\left\lfloor\frac{2-\lam_2}{1-\lam_1\lam_2}\right\rfloor,\ \ n_1=\left\lfloor\frac{1-2\lam_1}{1-\lam_1\lam_2}\right\rfloor+1
\een
to guarantee ${\rm dist}(\cP_1,\cP_2)\ge 2$, where for any $\gamma\in\R$, $\lfloor \gamma\rfloor$ is the integer strictly less than $\gamma$. Similarly, in the case displayed in Fig.\ref{fig:x2}(right), one can take
\ben
m_2=\left\lfloor\frac{1-2\lam_2}{1-\lam_1\lam_2}\right\rfloor+1,\ \ n_1=\left\lfloor\frac{2-\lam_1}{1-\lam_1\lam_2}\right\rfloor
\een
to guarantee ${\rm dist}(\cP_1,\cP_2)\ge 2$. This shows the lemma when $k_1,k_2<0$. The other cases when either one of $k_1,k_2$ or both $k_1,k_2$ are non-negative can be proved analogously. In fact, if $k_1\ge 0, k_2\ge 0$, then $m_1=\lfloor 2\lam_2\rfloor+2, m_2=1, n_1=1, n_2=\lfloor 2\lam_1\rfloor+3$; if $k_1\ge 0,k_2< 0$, then $m_1=2, m_2=\lfloor 2(-\lam_2)^+\rfloor+2, n_1=1, n_2=\lfloor (m_2+1)\lam_1\rfloor+3$; and if $k_1< 0,k_2\ge 0$, then $m_1=\lfloor (n_1+1)\lam_2\rfloor+2,m_2=1,n_1=\lfloor 2(-\lam_1)^+\rfloor+2$, $n_2=2$.

The macro-element is a large element with $\de_{M(K)}, \widetilde\de_{M(K)}\ge\min(1/(m_1+m_2),1/(n_1+n_2))$ by Definition \ref{def:2.2} of large elements. We note that $m_i,n_i$, $i=1,2$, are fixed constants, which depend only on the angle of the sector $S$ at vertex but are independent of $h_K$. This completes the proof when $|k_1|\le h_2/h_1$, $|k_2|>h_2/h_1$.

For the other cases, we just give the macro-element $M(K)$ and omit the details of the proof. By the symmetry, we only need to consider two cases. The first is when $L_1,L_2$ tend to infinity in the same region, for example in the region ${\rm IV}$. Then $|k_i|>h_2/h_1$, $i=1,2$, so that $\lam_i=k_i^{-1}h_2/h_1$ satisfies $|\lam_i|<1$, $i=1,2$. The integers $m_i,n_i$, $i=1,2$, in the macro-element $M(K)$ are defined by
\begin{align*}
&m_1=\max_{i=1,2}\lfloor (n_1+1)\lam_i^+\rfloor+3,\ \
m_2=\max_{i=1,2}\lfloor (n_1+1)(-\lam_i)^+\rfloor+2,\\
&n_1=\left\lfloor\frac{3}{|\lam_1-\lam_2|}\right\rfloor+1,\ \ n_2=2.
\end{align*}
It is clear that $M(K)$ is a large element with $\de_{M(K)},\widetilde\de_{M(K)}\ge\min(1/(m_1+m_2),1/(n_1+n_2))$.

The second case is when $L_1$ and $L_2$ tend to infinity in the opposite regions, for example, $L_1$ in the region ${\rm II}$ and $L_2$ in the region ${\rm IV}$. Then the integers $m_i,n_i$, $i=1,2$, in the macro-element $M(K)$ are defined by
\begin{align*}
m_1=\max_{i=1,2}\lfloor 2\lam_i^+\rfloor+3,\ \
m_2=\max_{i=1,2}\lfloor 2(-\lam_i)^+\rfloor+2,\ \ n_1=1,\ \ n_2=2.
\end{align*}
It is obvious that $M(K)$ is a large element with $\de_{M(K)},\widetilde\de_{M(K)}\ge\min(1/(m_1+m_2),1/(n_1+n_2))$.
\end{proof}

Notice that the macro-element $M(K)$ constructed in Lemma \ref{lem:3.1} is associated with the mesh $\cT$. The following lemma shows that the macro-element constructed is also nested with respect to the quad-refinements of the mesh. This fact is important for us to construct admissible subchains of interface or boundary elements inside which the interface or boundary is smooth. The proof is rather straightforward and we omit the details.

\begin{lem}\label{lem:3.2}
Let $S$ be a sector bounded by two half lines $L_1,L_2$ with vertex $Q\in K$, $K\in\cT$. Let $M(K)$ be the macro-element constructed in Lemma \ref{lem:3.1}. If one refines all elements in $M(K)$ by quad refinement to obtain a new mesh $\cT'$ and $K'\in\cT'$ is the singular element containing $Q$, then the macro-element $M(K')$ constructed on the mesh $\cT'$ satisfies $M(K')\subset M(K)$ and the element singular index $\widetilde\de_{M(K')}(\cT')=\widetilde\de_{M(K)}(\cT)$. Moreover, the interface or boundary elements of $\cT'$ included in $M(K)\backslash M(K')$ consist of either $\cT_2$ or two neighboring $\cT_1$ elements and thus can be merged to generate large elements over the mesh $\cT'$ inside $M(K)$, see Fig.\ref{fig:x3}.
\end{lem}

An important property of the singular pattern is that it stays unchanged if its outlet patch requires refinement due to the refinement of its neighboring elements in the chain of elements connecting to another singular pattern. Let $\cP$ be an outlet patch of $M(K)$ that is connected to another singular pattern by a chain $\mathfrak{G}$. If the neighboring elements of $\cP$ in the chain $\mathfrak{G}$ are refined, one can quad-refine the elements in $\cP$ and their neighboring elements in $\cS(M(K))_1$ so that the interface or boundary in $\cP$ can be covered by two large elements $M_1,M_2$, see Fig.\ref{fig:x3}. If the neighboring elements of $\cP$ in $\mathfrak{G}$ are further refined, one needs only to quad-refine the elements in $M_2$ and possibly their neighboring elements inside $\cS(M(K))_1$ so that the interface or boundary inside $\cP$ can be covered by large elements $M_1, M_{21},M_{22}$. This process of refining the outlet patches can be continued until the elements in the chain $\mathfrak{G}$ are constructed by \cite[Algorithm 6]{Chen2023JCP}. We summarized the algorithm for refining the outlet patch $\cP$ as the following algorithm. In the following, we define the level of the element $L(K)=0$ $\forall K\in \cT_0$.  When $K\in\cT$ is refined by quad refinement to four sub-elements $K_i$, we set $L(K_i)=L(K)+1$, $i=1,2,3,4$. Then we can use the level of an element to denote the size of the element.

\noindent\rule{\textwidth}{0.35mm}
\noindent{\bf{Algorithm 1:}} The algorithm for refining an outlet patch.

\vspace{-0.2cm}
\noindent\rule{\textwidth}{0.35mm}

{\bf{Input:}} {A singular pattern $M(K)$, an outlet patch $\cP$, and the difference of the level of elements $\ell=L(K')-L(K'')$, $K'\in\cP$, $K''\in(\cM^\Ga\cup\cM^\Sigma)\setminus M(K)$ neighboring to $K'$.

{\bf{Output:}} {The refined outlet patch $\cP'$ and a set of large elements $\mathcal{L}(K)$ covering the interface or boundary elements inside $\cP\setminus \cP'$.}

$1^\circ$ Set $i=1$, $\cP_1=\cP$.

$2^\circ$ For $i=1,\cdots,\ell$, do

(i) Quad-refine elements in $\cP_i$ and their neighboring elements inside $\cS(M(K))_1\setminus M(K)$, which are not merged yet, to generate a new mesh, set $\cP_{i+1}$ as the set of interface or boundary elements neighboring $\pa \cS(M(K))_1$.

(ii) Merge the interface or boundary elements in $\cP_i\backslash\cP_{i+1}$ with their neighboring elements to generate a large element $M$. Add $M$ to $\mathcal{L}(K)$.

$3^\circ$ Set $\cP'=\cP_{\ell+1}$.

\vspace{-0.2cm}
\noindent\rule{\textwidth}{0.35mm}

\begin{figure}[!ht]
\centering
\includegraphics[width=0.8\textwidth]{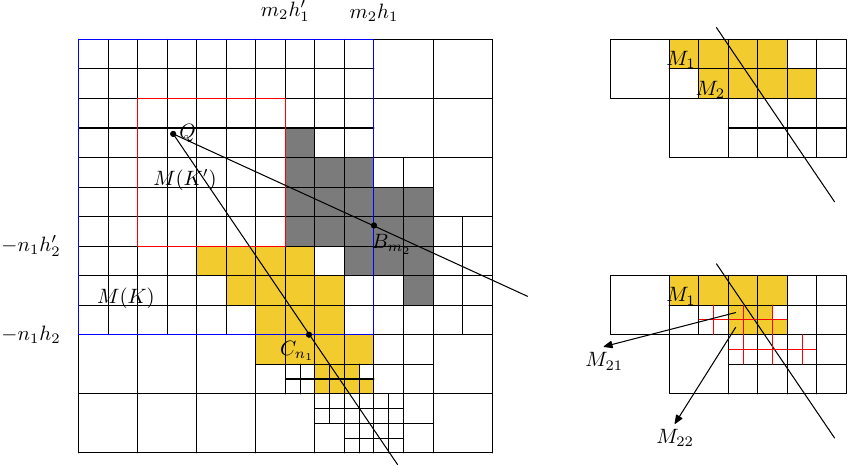}
\caption{The refinement of a singular pattern and its outlet patch, where $h_i'=h_i/2$, $i=1,2$. The macro-element $M(K')=(-2h_1',m_2h_1')\times(-n_1h_2',2h_2)$, $M(K)=(-2h_1,m_2h_1)\times (-n_1h_2,2h_2)$.}\label{fig:x3}
\end{figure}

We formulate the process in the following algorithm for refining a singular pattern if its outlet patch $\cP$ is refined.

\noindent\rule{\textwidth}{0.35mm}
\noindent{\bf{Algorithm 2:}} The merging algorithm for refining a singular pattern.

\vspace{-0.2cm}
\noindent\rule{\textwidth}{0.35mm}

{\bf{Input:}} {A singular pattern $M(K)$.}

{\bf{Output:}} {A locally refined mesh over $M(K)$ and a set of large elements $\mathcal{L}(K)$ covering the interface or boundary inside $M(K)$.}

$1^\circ$ Quad-refine elements in $M(K)$ to obtain a new mesh, find the singular element $K'$ on the new mesh, construct the singular pattern $M(K')$.

$2^\circ$ Merge the elements to generate large elements $N_i, i=1,\cdots, n_K$, which cover the boundary or interface inside $M(K)\bs M(K')$. Add $M(K')$ and $N_i,i=1,\cdots,n_K$, to $\mathcal{L}(K)$.

\vspace{-0.2cm}
\noindent\rule{\textwidth}{0.35mm}

\subsection{The merging algorithm}

After dealing with the singular elements, the remaining interface or boundary elements consist of several chains of elements inside which the interface or boundary is $C^2$-smooth.
Let $\mathfrak{G}$ be a chain of interface or boundary elements connecting two singular patterns $M(K_s)$, $M(K_e)$ with the associated outlet patch $\cP_s$, $\cP_e$, respectively. By definition, the start patch $\cP_s$ and the end patch $\cP_e$ are one of $\cP_i$, $i=1,2$, where $\cP_1=\{G\}, G\in\cT_2$, and $\cP_2=\{G_1,G_2\}, G_1, G_2\in\cT_1$.

One way to construct a locally induced mesh of $\mathfrak{G}$ is to refine the elements in $\cS(K)_2,K\in\mathfrak{G}$, to a uniform size so that the refined interface or boundary elements consist of an admissible chain $\mathfrak{G}'$ and thus one can use the merging algorithm in \cite[Algorithm 6]{Chen2023JCP}. However, it may do too many extra refinements to be suitable for the adaptive finite element algorithm. Our approach is to construct admissible subchains of $\mathfrak{G}$ with possible different element sizes in different subchains. In addition, the start and end patch of each admissible subchain belong to $\cP_1,\cP_2$ so that \cite[Algorithm 6]{Chen2023JCP}  can be used to generate the locally induced meshes of the subchain.

For a given chain $\mathfrak{G}$, let $\mathcal{G}\in \mathbb{R}^{n_a\times 2}$ be a two-dimensional array so that for $i=1,\ldots, n_a$, $\mathcal{G}(i,1)$, $\mathcal{G}(i,2)$ are the index of the start and end element of the $i$th admissible subchain in $\mathfrak{G}$. We use the following algorithm to find $\mathcal{G}$.

\noindent\rule{\textwidth}{0.35mm}
\noindent{\bf{Algorithm 3:}} Find all admissible subchains connecting two singular patterns $M(K_s)$, $M(K_e)$.

\vspace{-0.2cm}
\noindent\rule{\textwidth}{0.35mm}

{\bf{Input:}} {A chain $\mathfrak{G}$ that connects singular patterns $M(K_s), M(K_e)$ associated with outlet patches $\cP_s, \cP_e$, respectively.}

{\bf{Output:}} {An updated chain $\mathfrak{G}$ consisting of admissible subchains covering the interface or boundary inside the elements in $\mathfrak{G}$, and an index matrix $\mathcal{G}\in \mathbb{R}^{n_a\times 2}$ of admissible subchains of $\mathfrak{G}$.}

$1^\circ$ Locally refine the elements in $\mathfrak{G}$ such that $\mathfrak{G}$ satisfies the rules $2,3,4$ in Definition
\ref{def:3.1}. During the refining processes, call the refinement procedure in \cite[\S6]{Bonito} such that
the mesh satisfies Assumption (H3). If the outlet patch $\cP_s$ or $\cP_e$ need to
be refined then call Algorithm 1. Set $i_s=1,n_a=0$.

$2^\circ$ Find the maximum number $i_e$ such that $\mathfrak{G}(i_s:i_e)$ have the same size.

$3^\circ$ {\bf{if}} $\mathfrak{G}(i_e)==\mathfrak{G}(end)$ {\bf{then}}

$\quad\quad$ Set $n_a=n_a+1$ and $\mathcal{G}(n_a,1:2)=[i_s,i_e]$.

%
%
%

$\quad\quad$ {\bf{RETURN}} $\mathcal{G}$ and $\mathfrak{G}$.

$\quad$ {\bf{end}}

$4^\circ$ {\bf{if}} $|L(\mathfrak{G}(i_e+1))-L(\mathfrak{G}(i_e))|\leq 2$ {\bf{then}}

$\quad\quad$ {\bf{if}} ($\mathfrak{G}(i_e)\in \cP_1$ or $\mathfrak{G}(i_e-1:i_e)\in \cP_2$)

$\quad\quad\quad$ and ($\mathfrak{G}(i_e+1)\in \cP_1$ or $\mathfrak{G}(i_e+1:i_e+2)\in \cP_2$) {\bf{then}}

$\quad\quad\quad$ Set $n_a=n_a+1$, $\mathcal{G}(n_a,1:2)=[i_s,i_e]$, $i_s=i_e+1$ and {\bf GOTO} $2^\circ$.

$\quad\quad$ {\bf{end}}

$\quad$ {\bf{else}}

$\quad\quad$ {\bf{if}} $L(\mathfrak{G}(i_e+1))<L(\mathfrak{G}(i_e))$ {\bf{then}}

$\quad\quad\quad$ {\bf{if}} $\mathfrak{G}(i_e+1)\in \cP_e$ {\bf{then}}

$\quad\quad\quad\quad$ Call Algorithm 1 to refine $\cP_e$ to get $\cP_e'$, set $\cP_e=\cP_e'$ and update the $\mathfrak{G}$ as a

$\quad\quad\quad\quad$ chain which connecting to $\cP_s$ and $\cP_e$. {\bf GOTO} $2^\circ$.

$\quad\quad\quad$ {\bf{else}} Refine $\mathfrak{G}(i_e+1)$, {\bf GOTO} $2^\circ$.

$\quad\quad\quad$ {\bf{end}}

$\quad\quad$ {\bf{else}}

$\quad\quad\quad$ {\bf{if}} $\mathfrak{G}(i_e)\in \cP_s$ {\bf{then}}

$\quad\quad\quad\quad$ Call Algorithm 1 to refine $\cP_s$ to get $\cP_s'$, set $\cP_s=\cP_s'$ and update the $\mathfrak{G}$ as a

$\quad\quad\quad\quad$ chain which connecting to $\cP_s$ and $\cP_e$. {\bf GOTO} $2^\circ$.

$\quad\quad\quad$ {\bf{else}}

$\quad\quad\quad$ Refine $\mathfrak{G}(i_e)$

$\quad\quad\quad$ {\bf{if}} $i_s==i_e$ {\bf{then}}

$\quad\quad\quad\quad$ Set $i_s=\mathcal{G}(n_a,1)$ and remove the $n_a$-th row of $\mathcal{G}$, $n_a=n_a-1$.

$\quad\quad\quad$ {\bf{end}}

%
%
%
%
%

$\quad\quad\quad$ {\bf GOTO} $2^\circ$.

$\quad\quad\quad$ {\bf{end}}

$\quad\quad$ {\bf{end}}

$\quad${\bf{end}}

\vspace{-0.2cm}
\noindent\rule{\textwidth}{0.35mm}


We remark that in the worst case, when all elements in $\mathfrak{G}$ are refined to be of the same size to consist of an admissible chain, Algorithm 3 will return one admissible chain, that is, $\mathcal{G}=[1,{\rm length}(\mathfrak{G})]$.

Our final merging algorithm combines the singular patterns and Algorithm 3 of finding admissible subchains to obtain an induced mesh whose elements cover the interface or boundary.

\noindent\rule{\textwidth}{0.35mm}
\noindent{\bf{Algorithm 4:}} The merging algorithm for a closed chain $\mathfrak{C}$.

\vspace{-0.2cm}
\noindent\rule{\textwidth}{0.35mm}

{\bf{Input:}} {A closed chain $\mathfrak{C}$.}

{\bf{Output:}} {The induced mesh Induced$(\mathfrak{\cT})$.}

$1^\circ$ Construct the singular patterns for each singular element, then the remainder interface or boundary elements will consist of some chains, each chain $\mathfrak{G}$ satisfying $\mathfrak{G}(1)\in \cP_1$ or $\mathfrak{G}(1:2)\in \cP_2$ and $\mathfrak{G}(end)\in \cP_1$ or $\mathfrak{G}(end-1:end)\in \cP_2$.

$2^\circ$ For each remainder chain $\mathfrak{G}$, call Algorithm 3 with input $\mathfrak{G}$ to obtain the index matrix $\mathcal{G}\in\mathbb{R}^{n_a\times 2}$.

$3^\circ$ For each index matrix, call \cite[Algorithm 6]{Chen2023JCP} with input $\mathfrak{G}(\mathcal{G}(i,1):\mathcal{G}(i,2)),i=1,\ldots,n_a$, to generate the locally induced mesh.

$4^\circ$ Construct new singular patterns by Algorithm 2 whose outlet patches are refined in steps $2^\circ$.

\vspace{-0.2cm}
\noindent\rule{\textwidth}{0.35mm}

From the definition of the singular patterns and the Theorem \ref{thm:3.1}, the above algorithm will successfully generate the induced mesh for the arbitrarily shaped piecewise smooth interface or boundary. Thus we have the following theorem.

\begin{thm}\label{thm:3.2}
Let $\delta_0=\min(1/5,\de_s)$, where $\de_s$ is the minimum element singular index of all singular elements. For any closed chain of interface or boundary elements $\mathfrak{C}$, Algorithm 4 will generate an induced mesh whose elements cover the interface or boundary included in the elements in $\mathfrak{C}$.
\end{thm}

\section{Numerical examples}
\label{sec_num}

In this section, we present some numerical examples to confirm our theoretical results. We use a posteriori error estimator in Theorem \ref{thm:2.1} to design the adaptive algorithm combined with our mesh merging algorithm. The hierarchical shape functions in Szab\'o and Babu\v{s}ka \cite[\S 5.3.2]{Szabo2011book} are used on quadrilateral elements. For shape functions on triangular elements, we refer to Adjerid et al. \cite{Adjerid2001} where the orthogonalization is applied to the face shape functions to reduce the condition number of the stiffness matrix. The hanging nodes are treated by modifying the connectivity mapping in Ainsworth and Senior \cite{Ainsworth}, in which they provide the efficient realization on non-uniform Cartesian meshes.

The generic form of our adaptive unfitted finite element method is given below.

\noindent\rule{\textwidth}{0.35mm}
\noindent {\bf Algorithm 5:} General form of an adaptive unfitted finite element method.

 \vspace{-0.2cm}
\noindent\rule{\textwidth}{0.35mm}
Given a tolerance $TOL>0$, the steering parameters $\eta_0\in (0,1/2),\gamma\in (0,1)$, and an initial Cartesian mesh $\cT_0$. Set $\ell=0$ and do

\bigskip
\noindent (1) Let the interface and boundary elements of $\cT_\ell$ form two chains $\mathfrak{C}$ and $\mathfrak{C}'$, respectively. Call the refinement procedure in \cite[\S 6]{Bonito} such that $\cT_\ell$ satisfies Assumption (H3). Use merging Algorithm 4 to generate an induced mesh $\cM_\ell={\rm Induced}(\cT_\ell)$.

\noindent (2) If the interface or boundary elements in $\cM_\ell$ do not satisfy Assumption (H2),  set $S=\{K'\in \cT_\ell:K'\subset K, \, \eta_{K} >\eta_0 \, \text{ for some } K\in \cM_\ell\}$, release all merged elements in $\mathfrak{C}$ or $\mathfrak{C}'$, refine elements in $S$ by quad refinements, {\bf GOTO} (1).

\noindent (3) Compute the unfitted finite element solution $U_\ell$ satisfying \eqref{a2} on the mesh $\cM_\ell$.

\noindent (4) Compute the a posteriori error indicators $\xi_K\ \forall K\in\cM_\ell$ and the global (a posteriori) error estimator $\mathscr{E}(U_\ell,\cM_\ell)=(\sum_{K\in \cM_\ell}\xi_K^2)^{1/2}$.

\noindent (5) If $\frac{\mathscr{E}(U_\ell,\cM_\ell)}{\mathscr{E}(U_0,\cM_0)} <TOL$ then  {\bf STOP}.

\noindent (6) Mark the elements in $\widehat{\cM}_\ell\subset \cM_\ell$ such that
\ben
\left(\sum_{K\in \widehat{\cM_\ell} }\xi_K^2\right)^{1/2}\geq \gamma \mathscr{E}(U_\ell,\cM_\ell).
\een

\noindent (7) Refine the elements in $\widehat{\cT}_\ell=\{K\in\cT_\ell: \overline{K}\cap \overline{K'} \neq \emptyset, \, K'\in \widehat{\cM}_\ell\}$ to obtain $\cT_{\ell+1}$.

\noindent (8) Set $\ell=\ell+1$, and {\bf GOTO} (1).

\vspace{-0.2cm}
\noindent\rule{\textwidth}{0.35mm}

In the following examples, we set $\eta_0=0.05$ and $\gamma=0.5$, unless otherwise stated. The algorithms are implemented in MATLAB on a workstation with Intel(R) Core(TM) i9-10885H CPU 2.40GHz and 64GB memory.

\begin{exmp}
Let $\Omega=\{(x_1,x_2):(\xi-\frac 12)^2+\eta^2 <1, (\xi+\frac 12)^2 +\eta ^2 <1\}$, where $\xi=\cos(\theta)x_1+\sin(\theta)x_2$, $\eta=-\sin(\theta)x+\cos(\theta)y$, and $\theta=\frac{2\pi}{5}$. The boundary of the domain has two singular points at $A=(-\frac{\sqrt{3}}{2}\sin(\theta), \frac{\sqrt{3}}{2}\cos(\theta))$, $B=(\frac{\sqrt{3}}{2}\sin(\theta), -\frac{\sqrt{3}}{2}\cos(\theta))$, see Fig.\ref{fig:4.1}(a). Let the coefficient $a=1$ in $\Om$, the source $f$ and boundary condition $g$ be given such that the exact solution is
\ben
u_1(x_1,x_2)=\frac 1{\sqrt{2}}\Re[(z^2+3/4)^{1/2}],\ \ \mbox{or}\ \
u_2(x_1,x_2)=\frac 1{\sqrt{2}}\Re[(z^2+3/4)^{1/2}]+\cos(20 x_1x_2),
\een
where $z=\xi+{\bf i}\eta$ and $\Re[w]$ denotes the real part of $w$. We choose the negative half real line in the complex plane as the branch cut of $z^{1/2}$. Notice that $(z^2+3/4)^{1/2}$ is analytic in $\Om$ but has singularities at $A,B$. Thus $\Re[(z^2+3/4)^{1/2}]$ is harmonic in $\Om$ but is singular at $A,B$.
\end{exmp}

Fig.\ref{comp_domain_exmp1_2}(a) shows the decay of the error $\| u_1-U_\ell\|_{\rm DG}$ for $p=1,2,3,4,5$, which is faster than the quasi-optimal decay $C N_{\ell}^{-\frac{p}{2}}$ for high-order methods. Here $N_\ell$ is the number of the degrees of freedom (\#DoFs) of the mesh $\cM_\ell$. This is because the exact solution $u_1$ is very smooth except at $A$, $B$. The error of high-order methods in the smooth region is much smaller than the error in the region near $A, B$. The adaptive algorithm always refines elements close to $A$, $B$. The asymptotic behavior of the quasi-optimal decay of the high-order methods has not been reached before the code stops when the element size near the singularity is less than $10^{-9}$. However,
Fig.\ref{comp_domain_exmp1_3}\,(a) shows the quasi-optimal decay of the error $\|u_2-U_\ell\|_{\rm DG}\approx CN_\ell^{-p/2}$ when the solution $u_2$ has high frequency structures.  Table \ref{tab_err_dofs_ex1} shows clearly that the error of high-order methods is much smaller than the low-order method when the \#DoFs are almost the same. Fig.\ref{comp_mesh} shows the computational meshes for $p=1,3,5$ with almost the same \#DoFs.

Fig.\ref{comp_domain_exmp1_2}(b) and Fig.\ref{comp_domain_exmp1_3}(b) depict the effectivity index
\ben
I_{\rm eff}=\mathscr{E}(U_\ell,\cM_{\ell})/\|u-U_\ell\|_{\rm DG},
\een
which ranges in $(3.5, 7.5)$. It shows that our a posteriori error estimates can effectively estimate the errors of the numerical solutions.

\begin{figure}[!ht]
\centering
\subfigure[]{\includegraphics[width=0.32\textwidth]{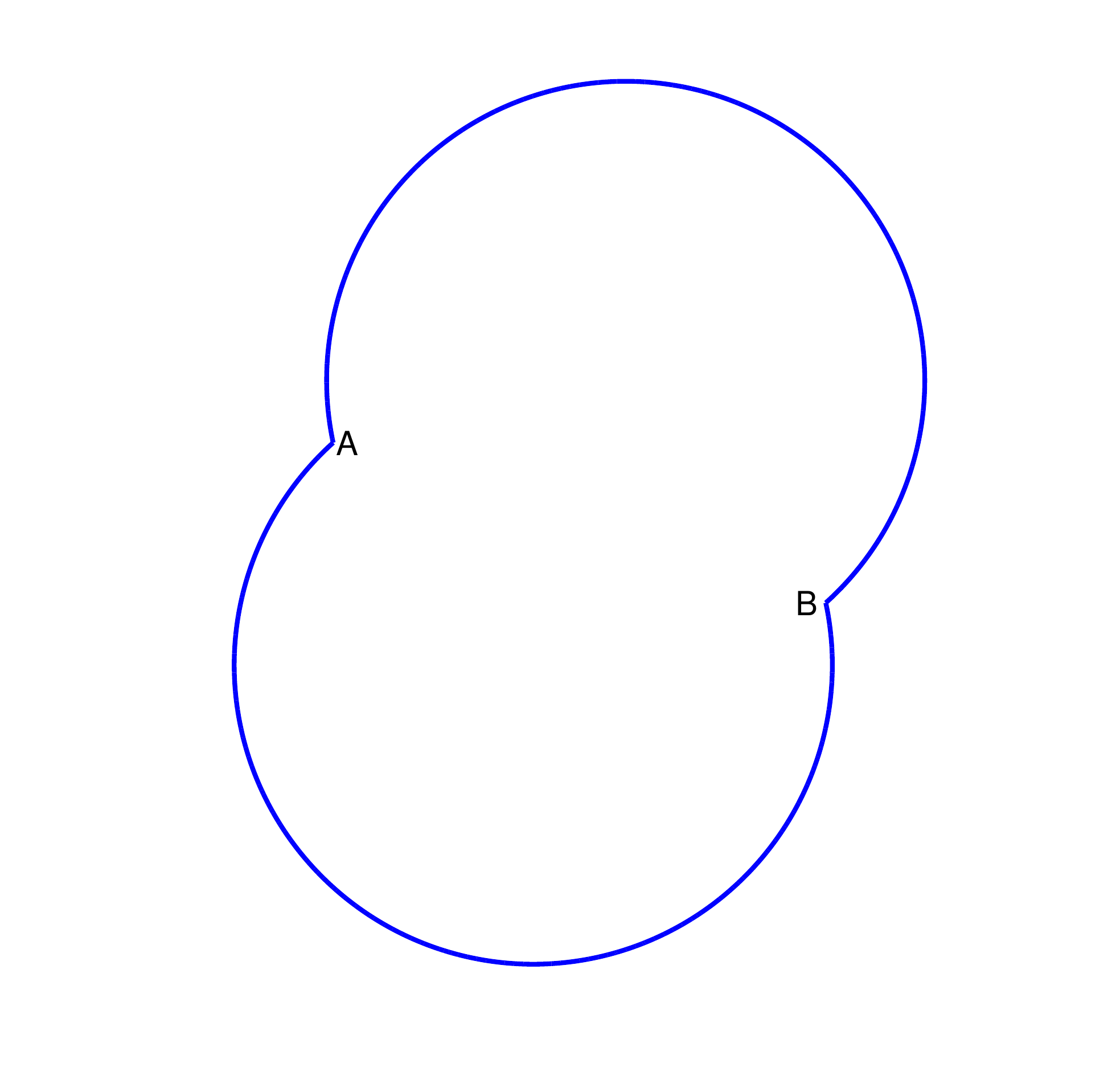}}
\subfigure[]{\includegraphics[width=0.4\textwidth]{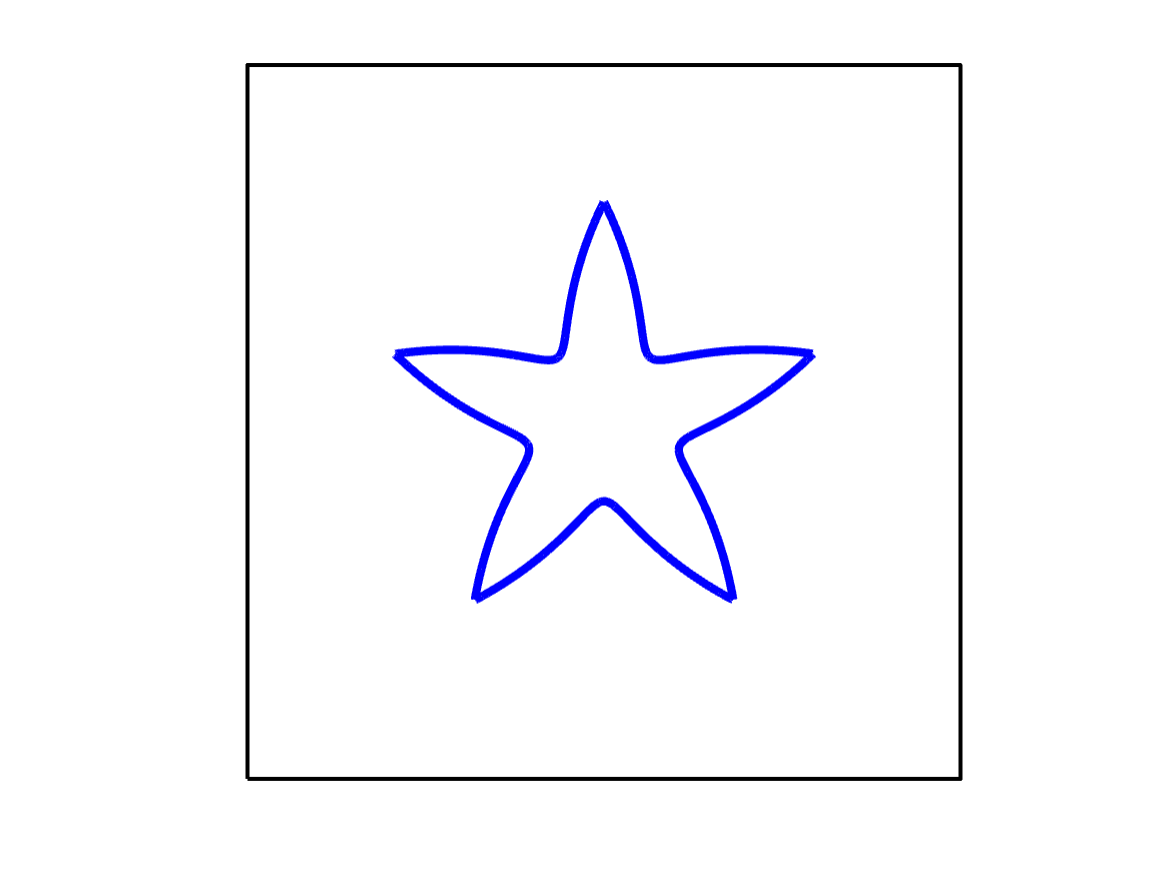}}
\caption{The computational domain of Example 1 (left) and Example 2 (right).}\label{fig:4.1}
\end{figure}

\begin{figure}[!ht]
\centering
\subfigure[]{\includegraphics[width=0.4\textwidth]{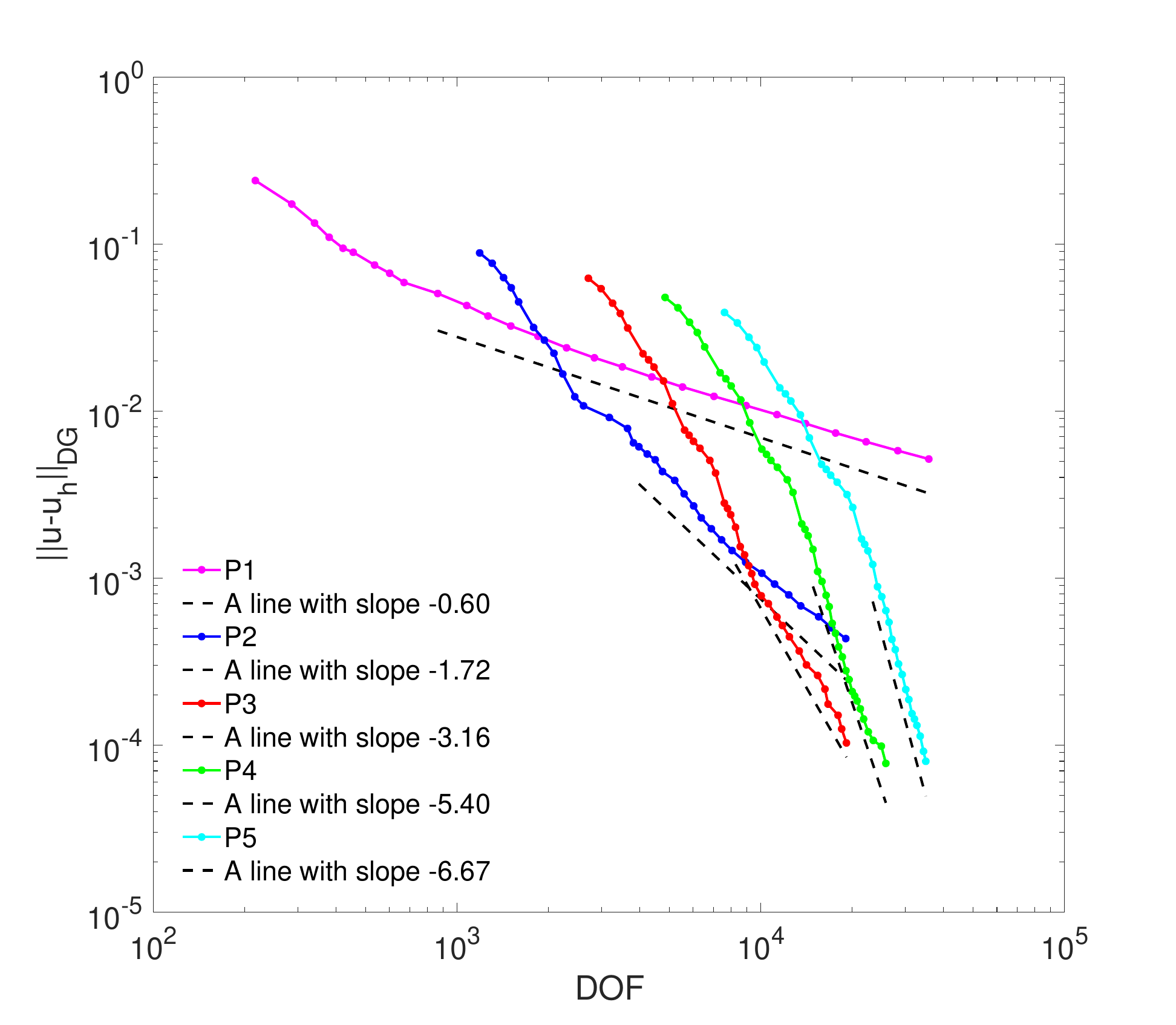}}
\subfigure[]{\includegraphics[width=0.45\textwidth]{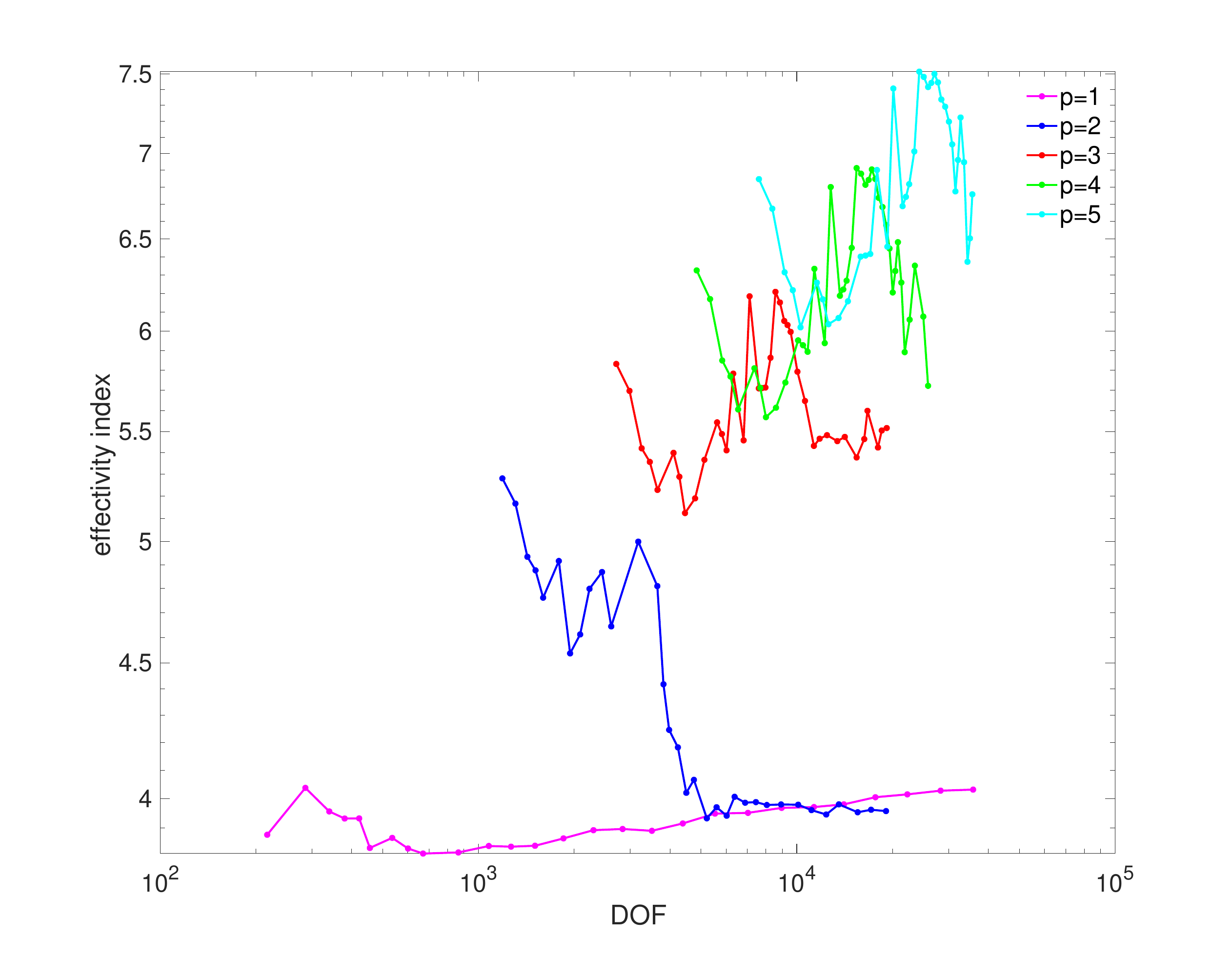}}
\caption{\label{comp_domain_exmp1_2}  Example 1: The exact solution is $u_1$. The quasi-optimal decay of the error of the adaptive solutions for $p=1,2,3,4,5$ (left) and the effectivity index against \#DoFs for $p=1,2,3,4,5$ (right).}
\end{figure}

\begin{figure}[!ht]
\centering
\subfigure[]{\includegraphics[width=0.4\textwidth]{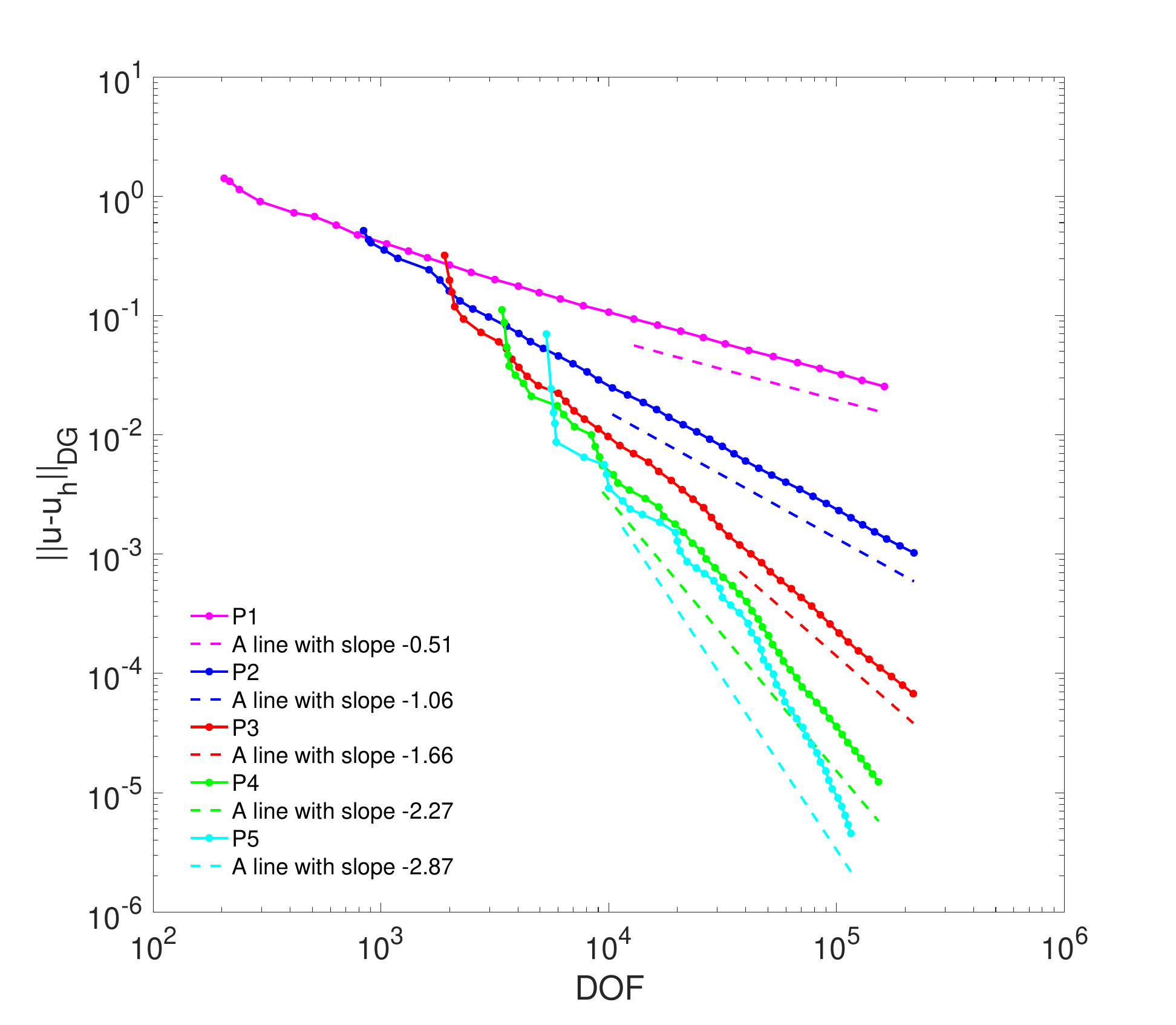}}
\subfigure[]{\includegraphics[width=0.45\textwidth]{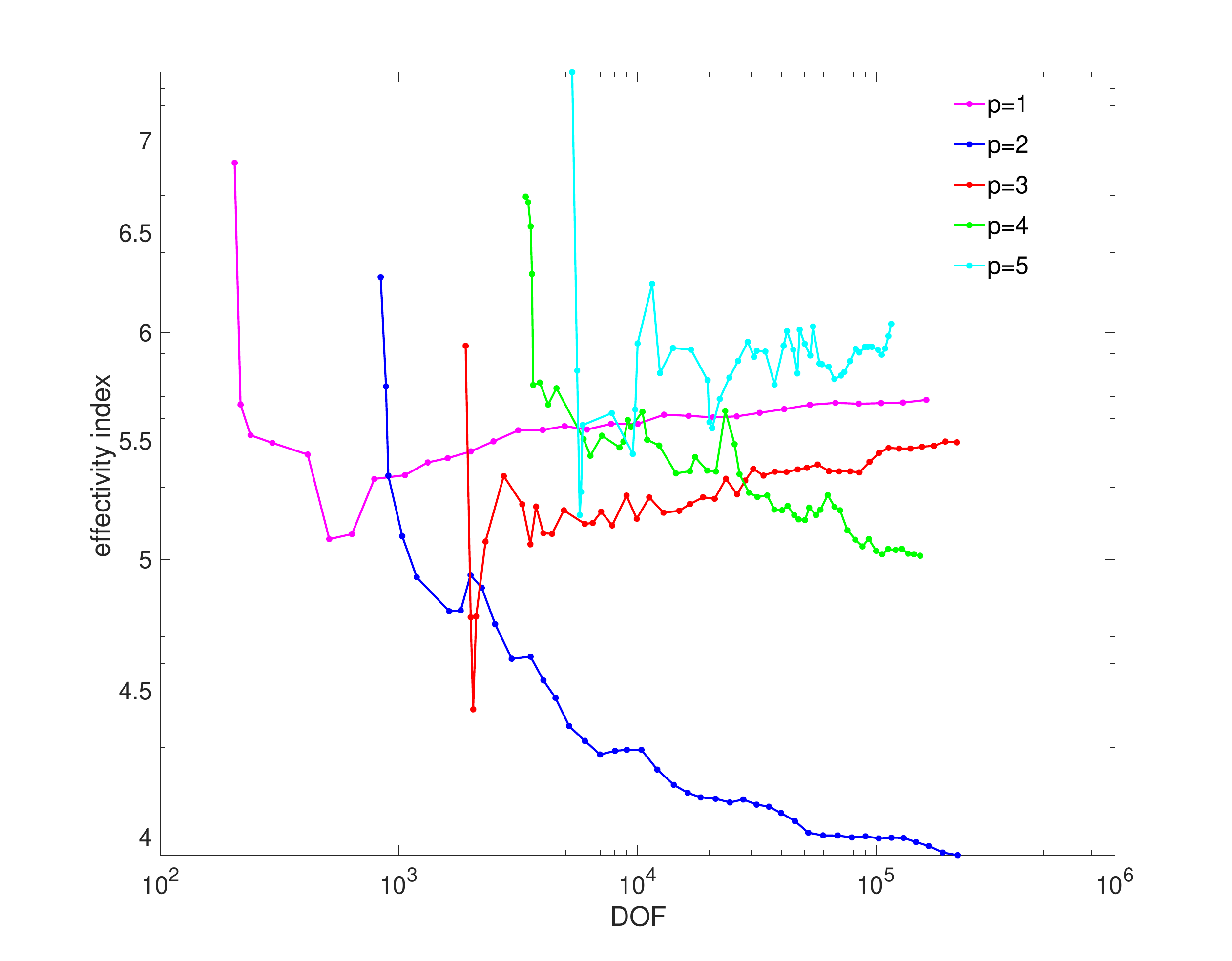}}
\caption{Example 1: The exact solution is $u_2$. The quasi-optimal decay of the error of the adaptive solutions for $p=1,2,3,4,5$ (left) and the effectivity index against \#DoFs for $p=1,2,3,4,5$ (right).}\label{comp_domain_exmp1_3}
\end{figure}

\begin{figure}[!ht]
\centering
\subfigure[]{\includegraphics[width=0.3\textwidth]{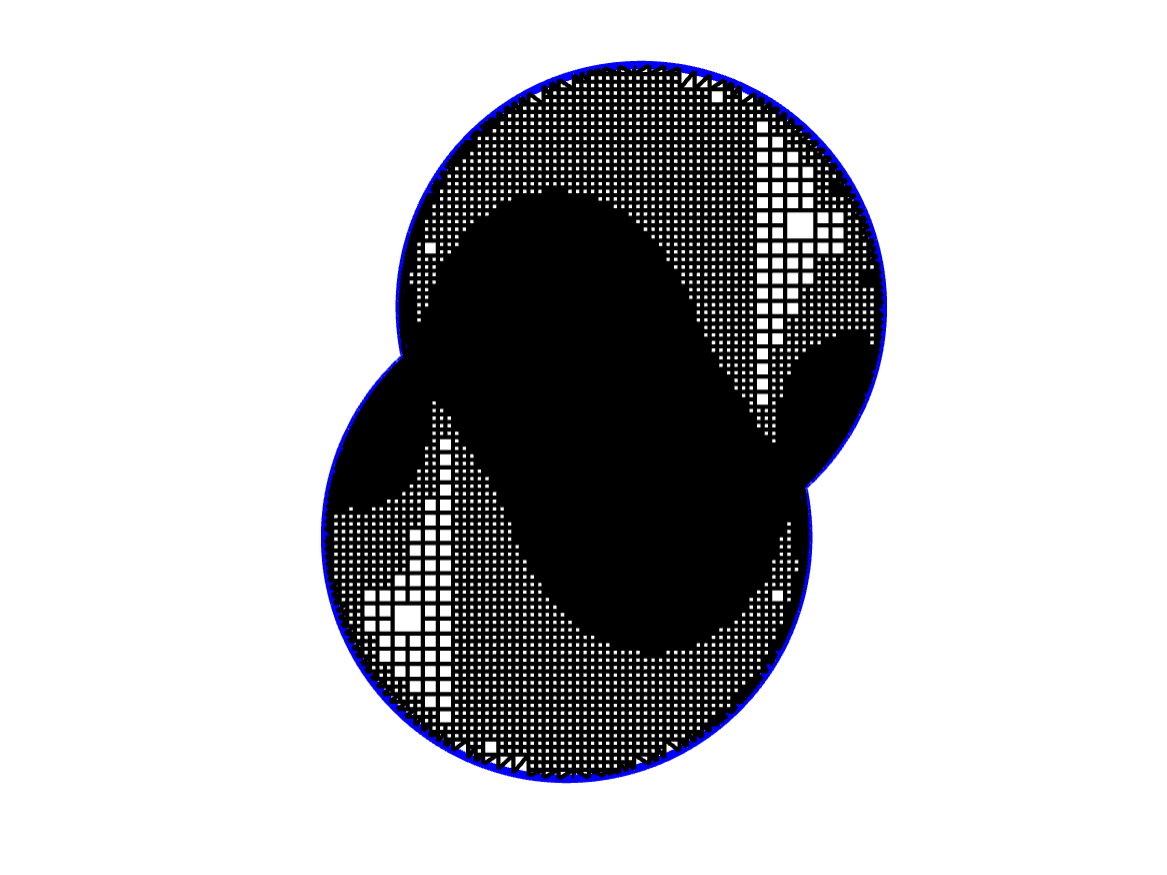}}
\subfigure[]{\includegraphics[width=0.3\textwidth]{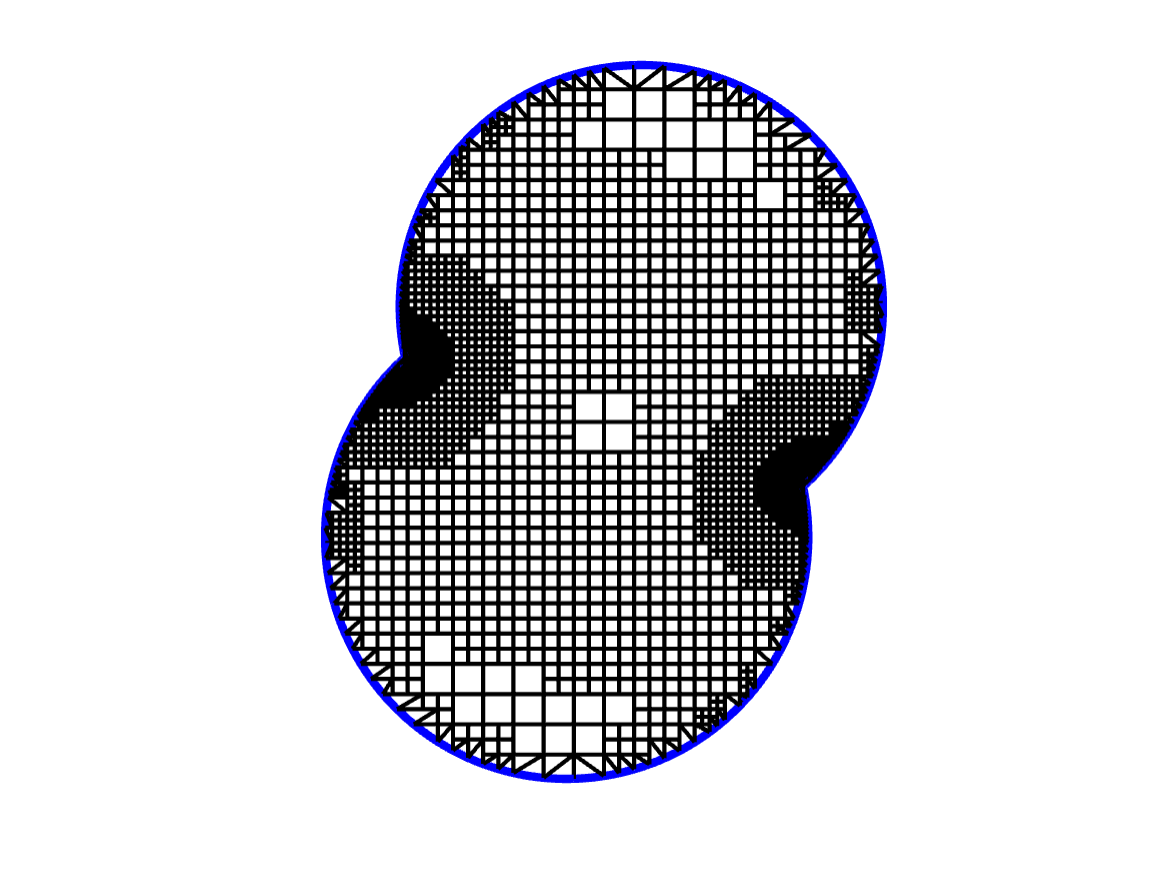}}
\subfigure[]{\includegraphics[width=0.3\textwidth]{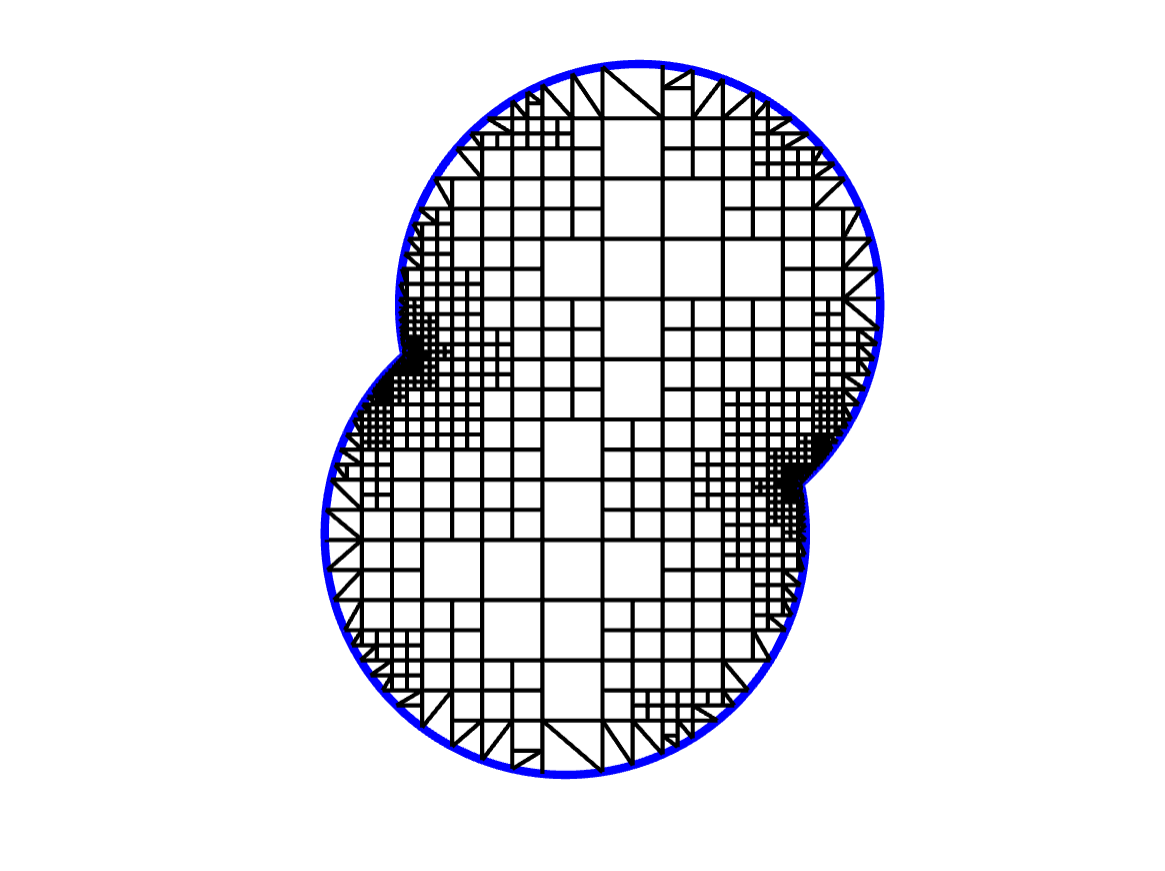}}
\caption{Example 1: The adaptive meshes for $p=1,2,3$ when the exact solution is $u_1$.} \label{comp_mesh}
\end{figure}

\begin{table}[!ht]\centering	
\begin{tabular}{|c|c|c|c|}
  \hline
  &$p=1$& $p=3$& $p=5$ \\ \hline
  $\#$DoFs& 52757& 51205 & 51891  \\ \hline
  Error & 4.51E-02 & 7.12E-03 & 1.29E-05  \\ \hline
 \end{tabular}
\caption{Example 1: The error of the numerical solutions and $\#$DoFs when the exact solution is $u_2$.}\label{tab_err_dofs_ex1}
\end{table}

\begin{exmp}
We consider a five-star shaped singular interface problem, see Fig.\ref{fig:4.1}\,(right). Let $\Omega_1=\{(r\cos(\theta),r\sin(\theta)): r<q(\theta), \, \theta\in [\frac{\pi}{10},\frac{21\pi}{10})\}$ and $\Om_2=(-2,2)^2\backslash\bar\Om_1$, where
\begin{align*}
q(\theta)=2\left(\theta-\frac{3\pi}{10}-\frac{2\pi}{5}j\right)^2+\frac{4}{9}, \quad \theta \in \left[\frac{\pi}{10}+\frac{2\pi}{5}j,\frac{\pi}{2}+\frac{2\pi}{5}j\right), \quad j=0,1,2,3,4.
\end{align*}
The coefficient $a$ in \eqref{m1}-\eqref{m2} is taken as $a_1=10$, $a_2=1$. We take the source $f_1=1$, $f_2=\cos(20x_1x_2)$, and boundary condition $g=0$.
\end{exmp}

\begin{figure}[!ht]
\centering
{\includegraphics[width=0.45\textwidth]{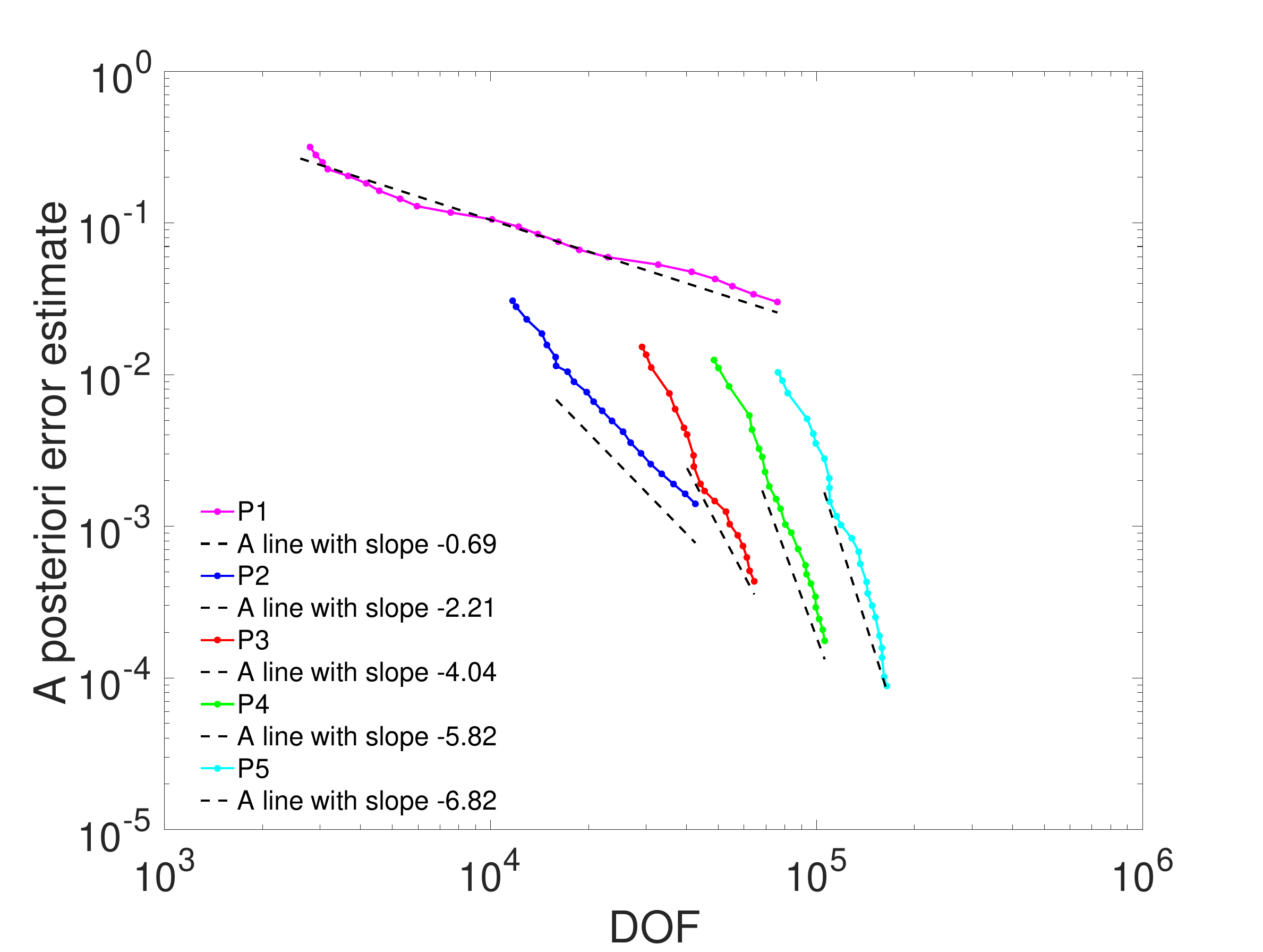}}
{\includegraphics[width=0.45\textwidth]{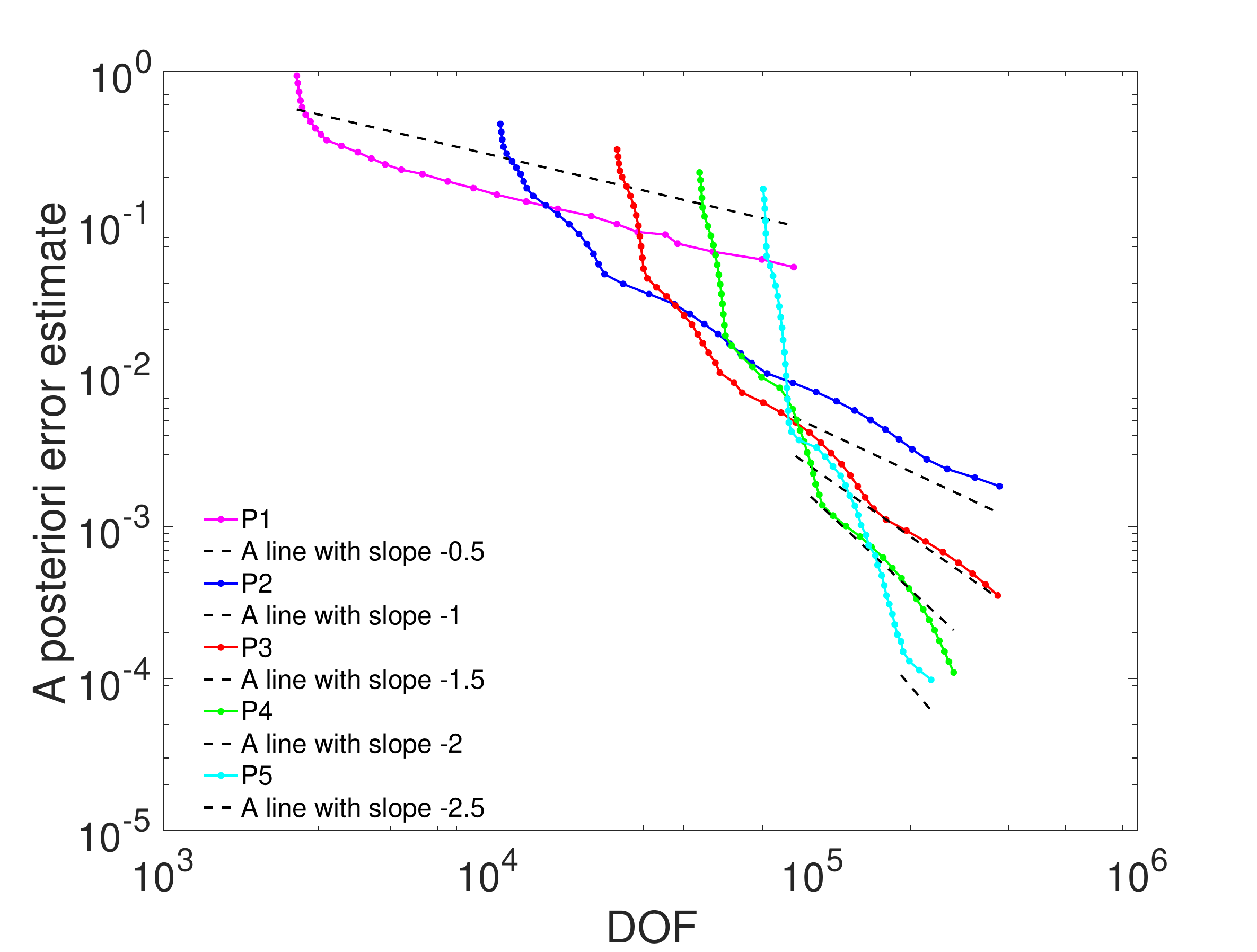}}
\caption{Example 2: The quasi-optimal decay of the error of the adaptive solutions for $p=1,2,3,4,5$ with $f_1=1$ (left) and $f_2=\cos(20x_1x_2)$ (right).}\label{comp_example2_five_star}
\end{figure}
\begin{figure}[!ht]
\centering
\subfigure[]{\includegraphics[width=0.3\textwidth]{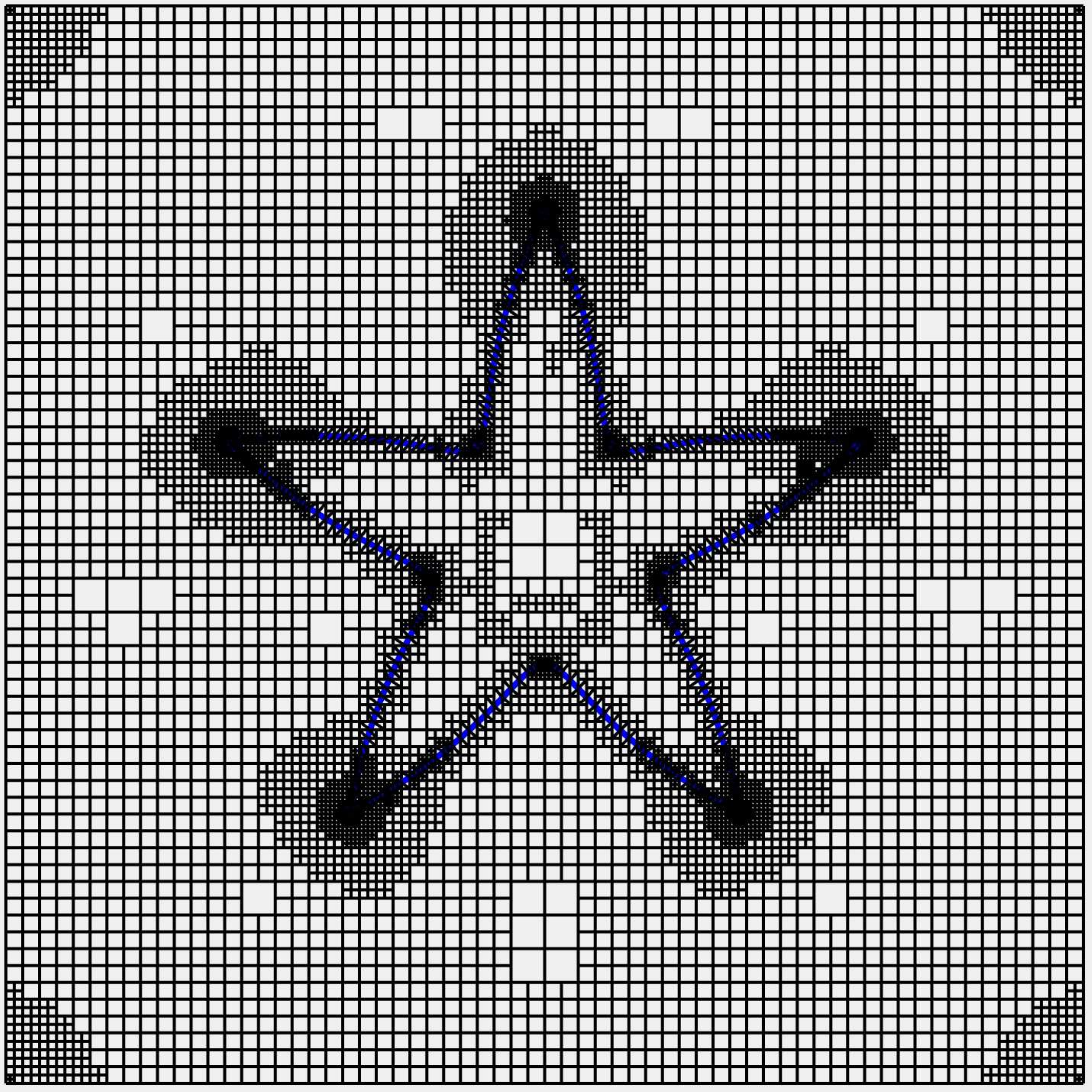}}
\subfigure[]{\includegraphics[width=0.3\textwidth]{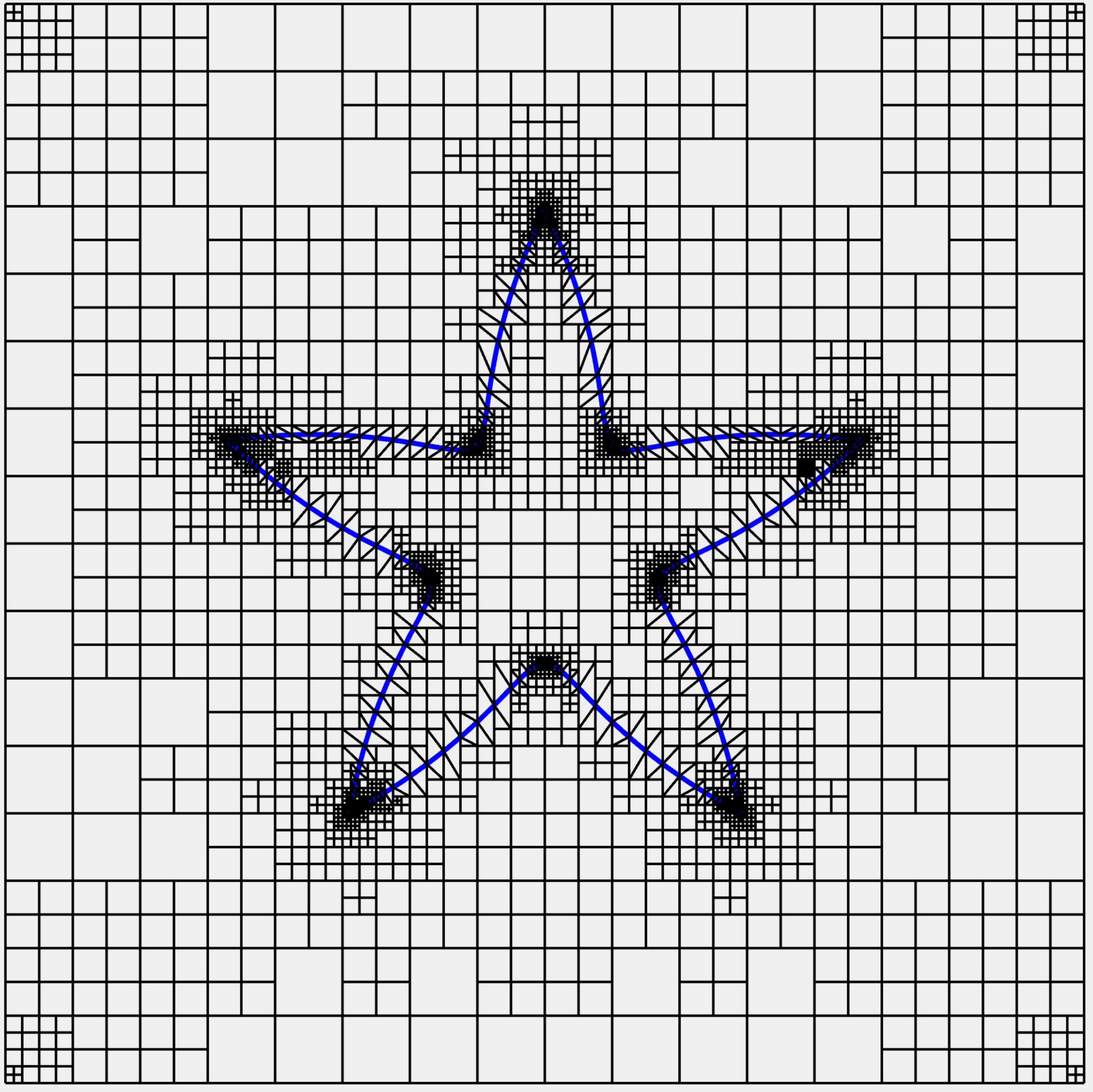}}
\subfigure[]{\includegraphics[width=0.3\textwidth]{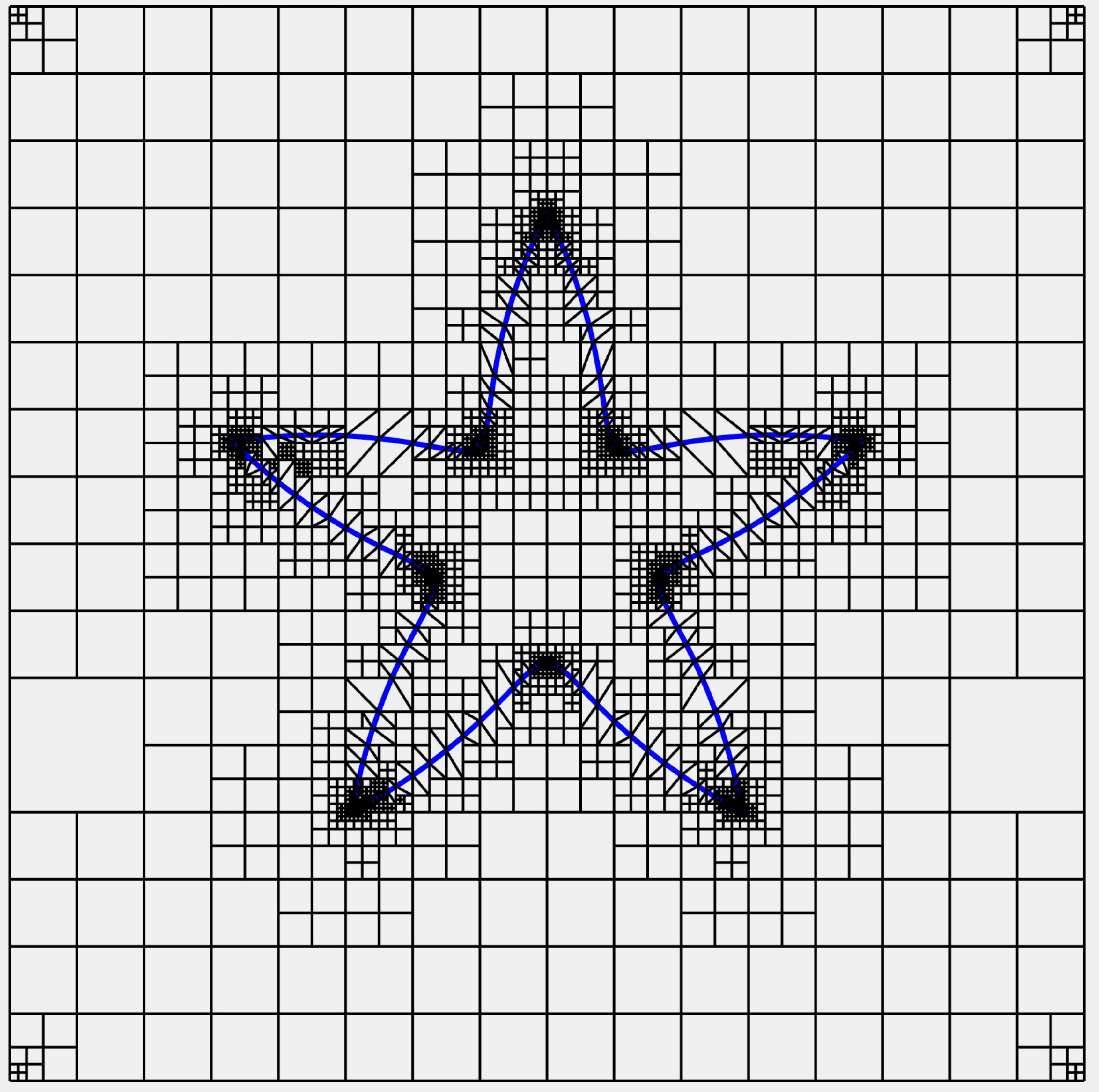}}
\caption{Example 2: The adaptive meshes for $p=2,3,4$ when $f_1=1$.} \label{comp_mesh_ex2}
\end{figure}

Fig.\ref{comp_example2_five_star} shows the quasi-optimal convergence the adaptive solutions when $f_1=1$ (left) and $f_2=\cos(20x_1x_2)$ (right). We again observe the decay is faster than the quasi-optimal decay of the high-order methods $p=2,3,4,5$, when $f_1=1$. On the other hand, we observe the quasi-optimal convergence rate of the adaptive solutions when $f_2=\cos(20x_1x_2)$ for $p=1,2,3,4,5$.

Fig.\ref{comp_mesh_ex2} shows the adaptive meshes for $p=2,3,4$. We observe that the meshes are mainly refined around the sharp corners where the solution is singular.  Table \ref{tab_err_dofs_ex2} shows the error of the high-order methods is much smaller than the low-order method when the \#DoFs are almost the same.

\begin{table}[!ht]\centering	
\begin{tabular}{|c|c|c|c|}
  \hline
  &$p=1$& $p=3$& $p=5$ \\ \hline
  $\#$DoFs& 253604 & 251406 & 231231  \\ \hline
  Error & 1.87E-02 & 8.00E-04 & 9.81E-05  \\ \hline
 \end{tabular}
\caption{Example 2: The error of the numerical solutions and $\#$DoFs when $f_2=\cos(20x_1x_2)$.}\label{tab_err_dofs_ex2}
\end{table}

\begin{exmp}\label{example3}
In this example, as an application of the adaptive high-order unfitted finite element method developed in this paper, we consider the influence of the geometric modeling error to the solution of elliptic problems. For $i=1,2$, let $u_i\in H^1(D_i)$ be the solution of the problem
\begin{align*}
-\Delta u_i&=1 \quad {\rm in} \, \, D_i,\quad u=0 \quad {\rm on}\, \, \partial D_i,
\end{align*}
where for $b_1=1/2, b_2=\sqrt 3/2$,
\begin{align*}
&D_1=\{(x_1,x_2)\in (-2,2)^2: \alpha |b_1 x_1 +b_2x_2|+\beta |-b_2 x_1+b_1 x_2|>1\},\\
&D_2=\{(x_1,x_2)\in (-2,2)^2: \alpha\sqrt{|b_1x_1 +b_2x_2|^2+\epsilon}+\beta\sqrt{|-b_2 x_1+b_1 x_2|^2+\epsilon}>1\}.
\end{align*}
We have $D_1\subset D_2$, see Fig.\ref{computational_domain_exmp3} when $\alpha=\sqrt{5}$, $\beta=\sqrt{2/3}$, and $\epsilon=0.001$. For $\epsilon>0$ small, $D_1$ can be viewed as an approximation of the domain $D_2$ with round-shaped corners.
\end{exmp}

The exact solutions of this example are unknown. We use Algorithm 3 to compute the numerical solutions $u_h$ on $D_1$ and $u_{h,\epsilon}$ on $D_2$. The tolerance is set to be $TOL$=1.0E-05 and the finite element approximation order is $p=5$. Fig.\ref{computational_mesh_Omega1}-\ref{computational_mesh_Omega2} depict the computational meshes.

It is easy to see that $|G_1|=|D_2\backslash\bar D_1|=O(\epsilon\log(1/\epsilon))$. The solution $u_1$ has reentrant corner singularities. It follows from the standard argument (see, e.g., Chen and Wu \cite[\S 4.1]{ChenWu}) that $u_1\in W^{1,p_1}(D_1)$, where $p_1=1/(2-(\pi/\theta))$, $\theta=2\pi-2\arctan(\sqrt{3/10})$. The solution $u_2$ is, on the other hand, smooth $u_2\in W^{1,\infty}(D_2)$. By Theorem \ref{thm:5.1}, $\|\na(u_1-u_2)\|_{L^2(D_1)}\le C(\epsilon\log(1/\epsilon))^\sigma$, $|\na(u_1-u_2)(x_*)|\le C(\epsilon\log(1/\epsilon))^{2\sigma}$, where $\sigma=(p_1-2)/(2p_1)\approx 0.297$, and $x_*=(\sqrt 2,\sqrt 2)$ that is away from the boundaries of $D_1$ and $D_2$. Fig.\ref{err_H1_interiori} shows the relative error of the numerical solutions
\begin{align*}
\frac{\|\nabla (u_{h}-u_{h,\epsilon})\|_{L^2(D_1)}}{\|\nabla u_h\|_{L^2(D_1)}}, \quad \frac{|\nabla (u_h-u_{h,\epsilon})(x_\ast)|}{|\nabla u_h(x_\ast)|},
\end{align*}
which clearly confirms our theoretical findings. We observe that the relative energy error in the whole domain $D_1$ is about $15\%$ when $\epsilon=10^{-3}$ but the relative gradient error is rather small away from the corner singularities. This implies that one should be careful if one is interested in the solution $u_2$ close to the tips of its round-shaped corners. The solution $u_1$ at these points may not give accurate approximations of $u_2$ there.

\begin{figure}[!ht]
\centering
\subfigure[]{\includegraphics[width=0.3\textwidth]{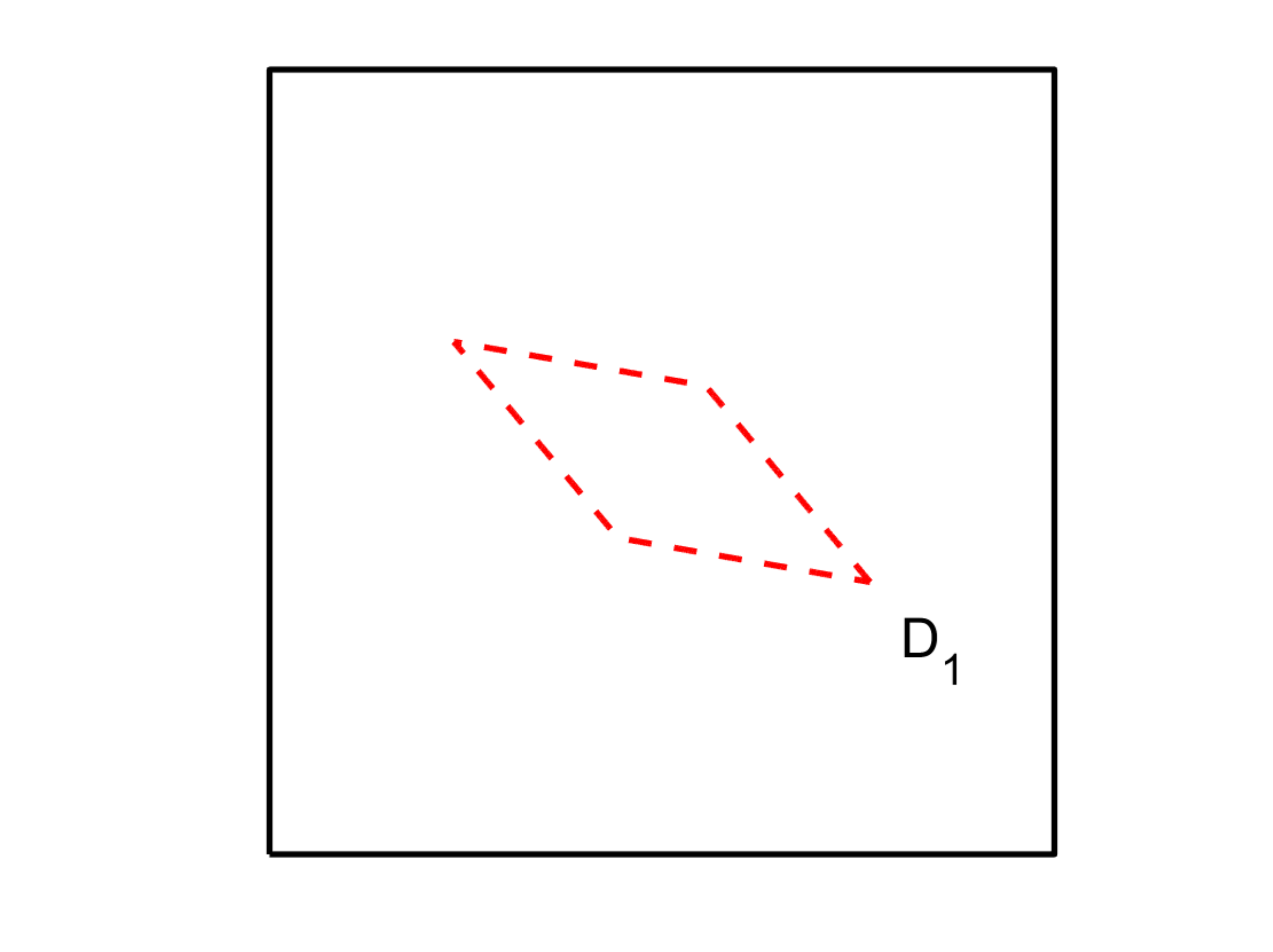}}
\subfigure[]{\includegraphics[width=0.3\textwidth]{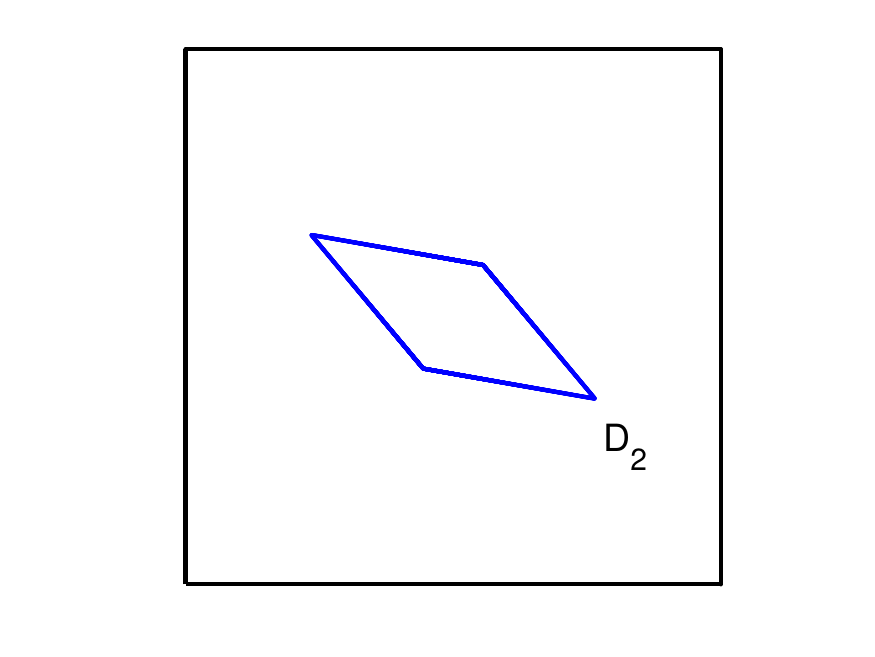}}
\subfigure[]{\includegraphics[width=0.3\textwidth]{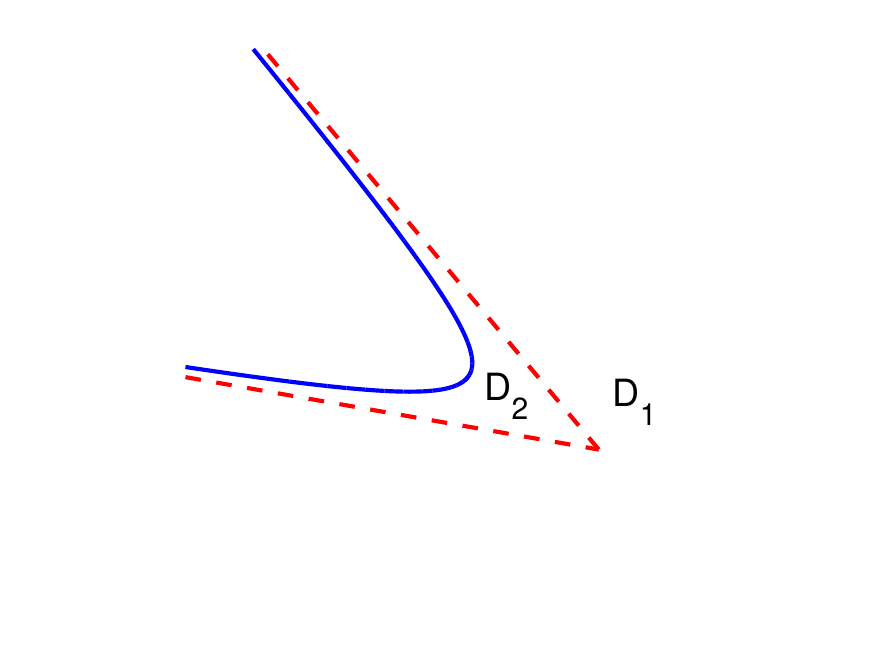}}
\caption{Example 3: The computational domain with the sharp corner $D_1$ (left), with the round shaped corner $D_2$ (middle), and the local domain near the singular point (right).} \label{computational_domain_exmp3}
\end{figure}

\begin{figure}[!ht]
\centering
\subfigure[]{\includegraphics[width=0.3\textwidth]{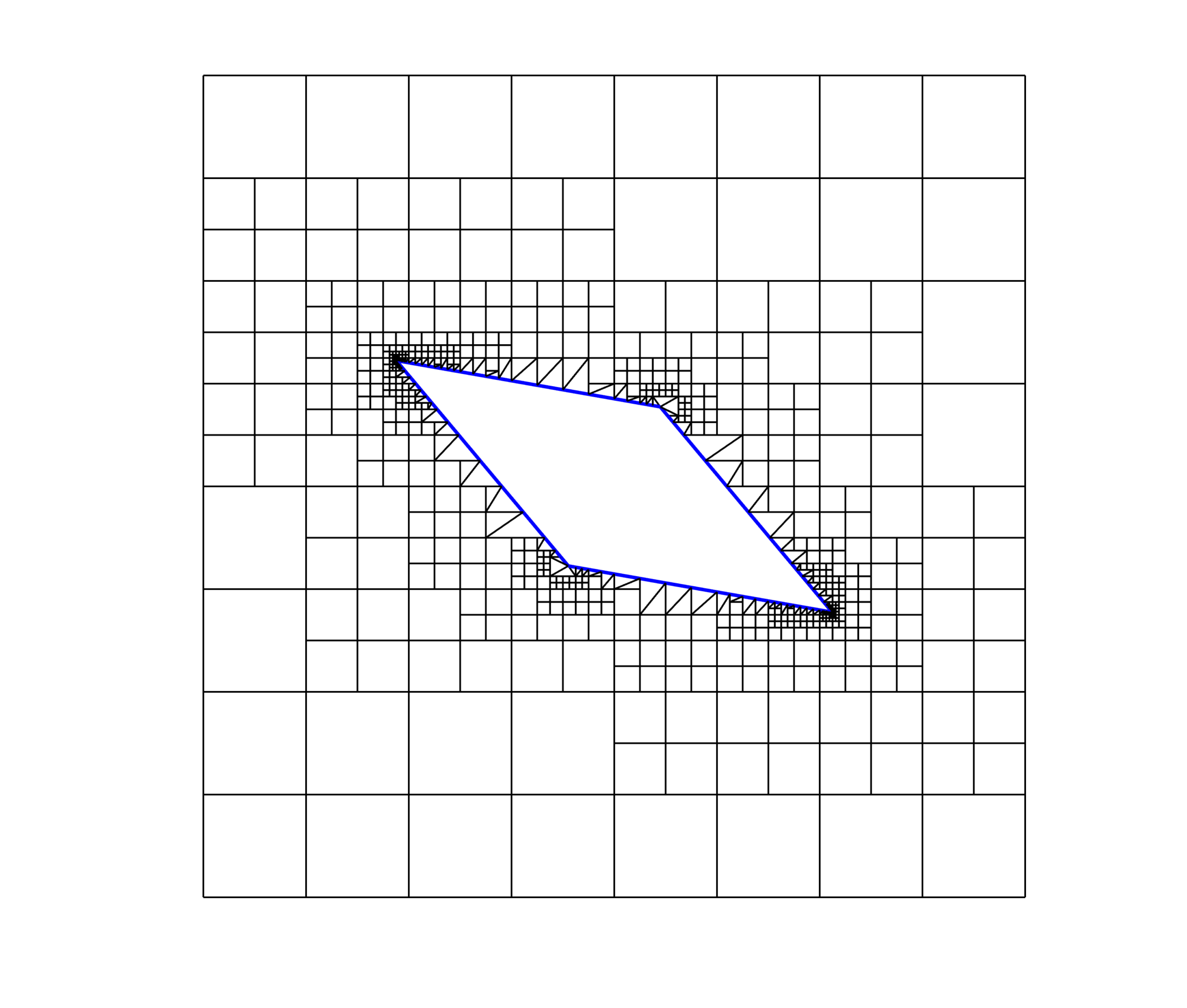}}
\subfigure[]{\includegraphics[width=0.3\textwidth]{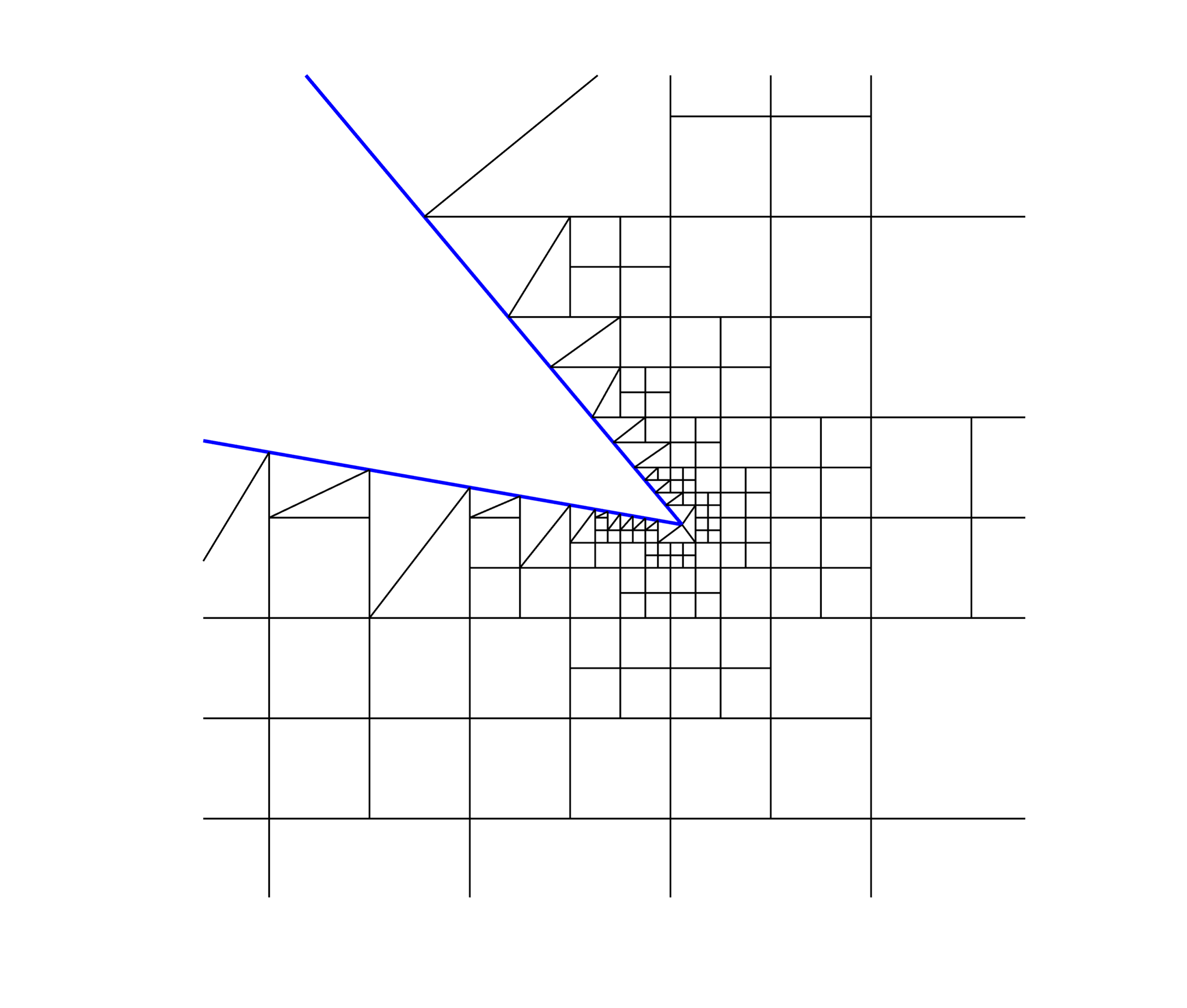}}
\caption{Example 3: The computational mesh in $D_1$ (left) and corresponding local mesh within $(1.056,1.064)\times(-0.616,-0.608)$ (right).} \label{computational_mesh_Omega1}
\end{figure}

\begin{figure}[!ht]
\centering
\subfigure[]{\includegraphics[width=0.3\textwidth]{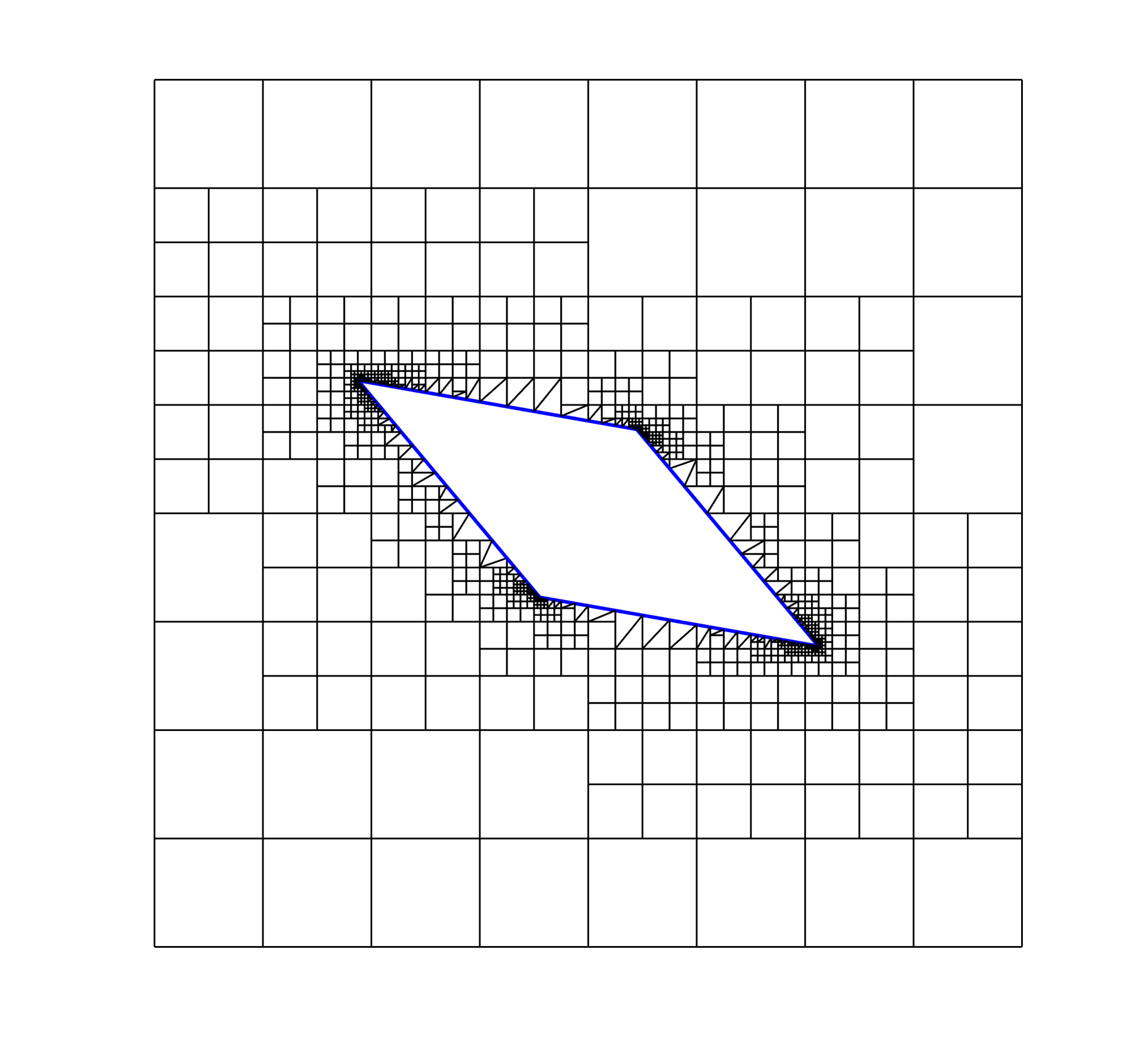}}
\subfigure[]{\includegraphics[width=0.3\textwidth]{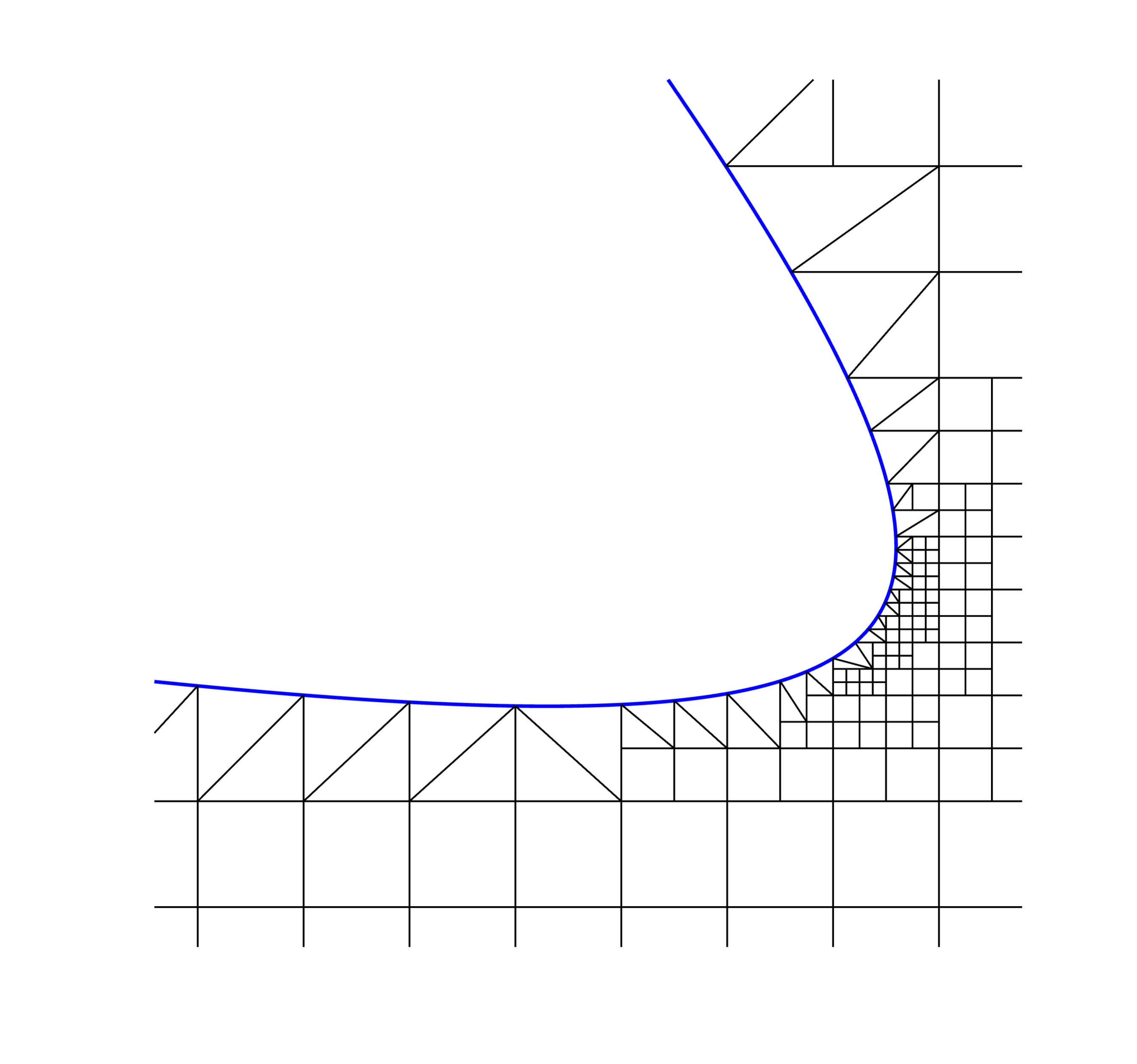}}
\caption{Example 3: The computational mesh in $D_2$ (left) and corresponding local mesh within $(1.054,1.058)\times(-0.612,-0.608)$ (right).} \label{computational_mesh_Omega2}
\end{figure}

\begin{figure}[!ht]
\centering
\subfigure[]{\includegraphics[width=0.44\textwidth]{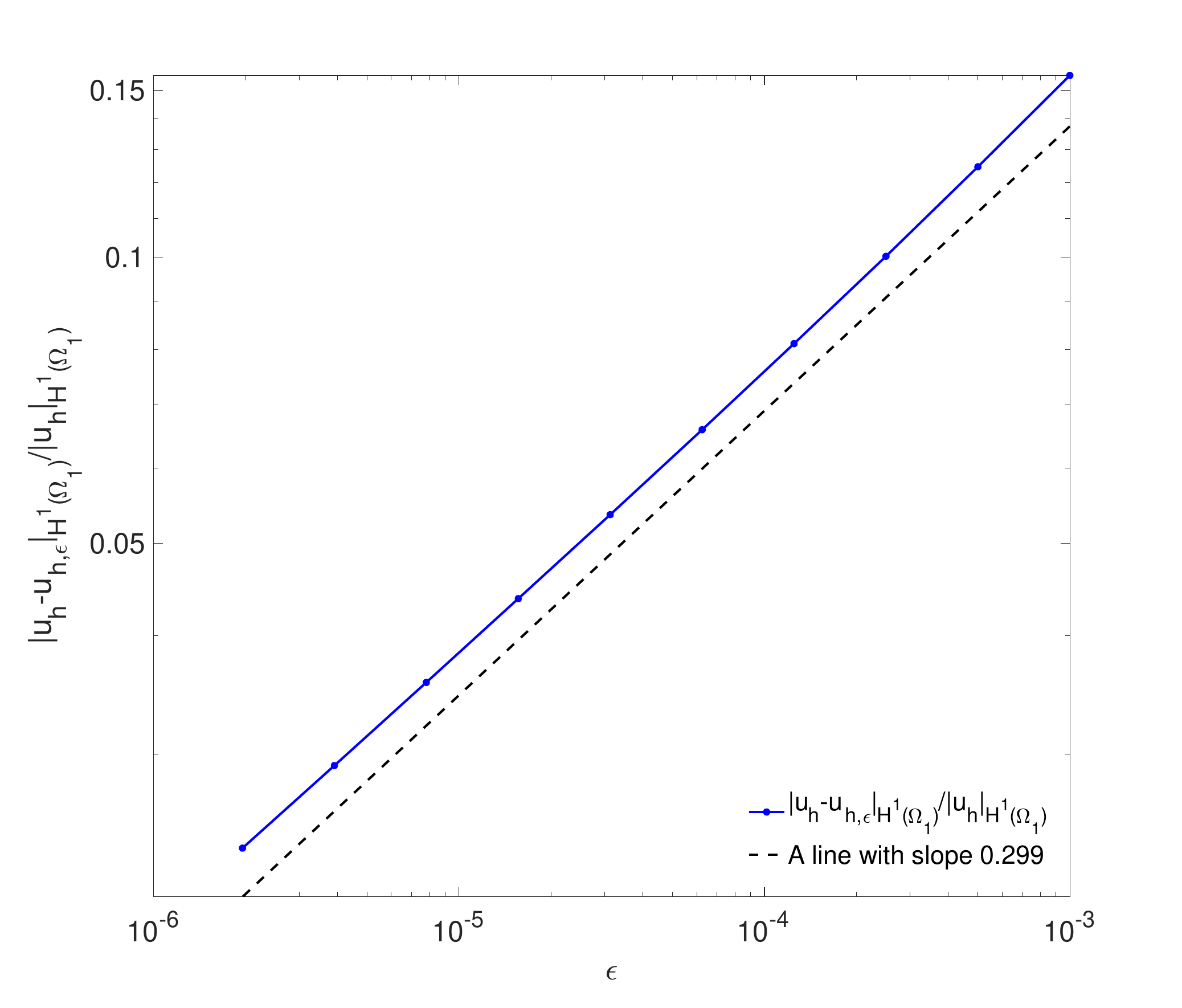}}
\subfigure[]{\includegraphics[width=0.4\textwidth]{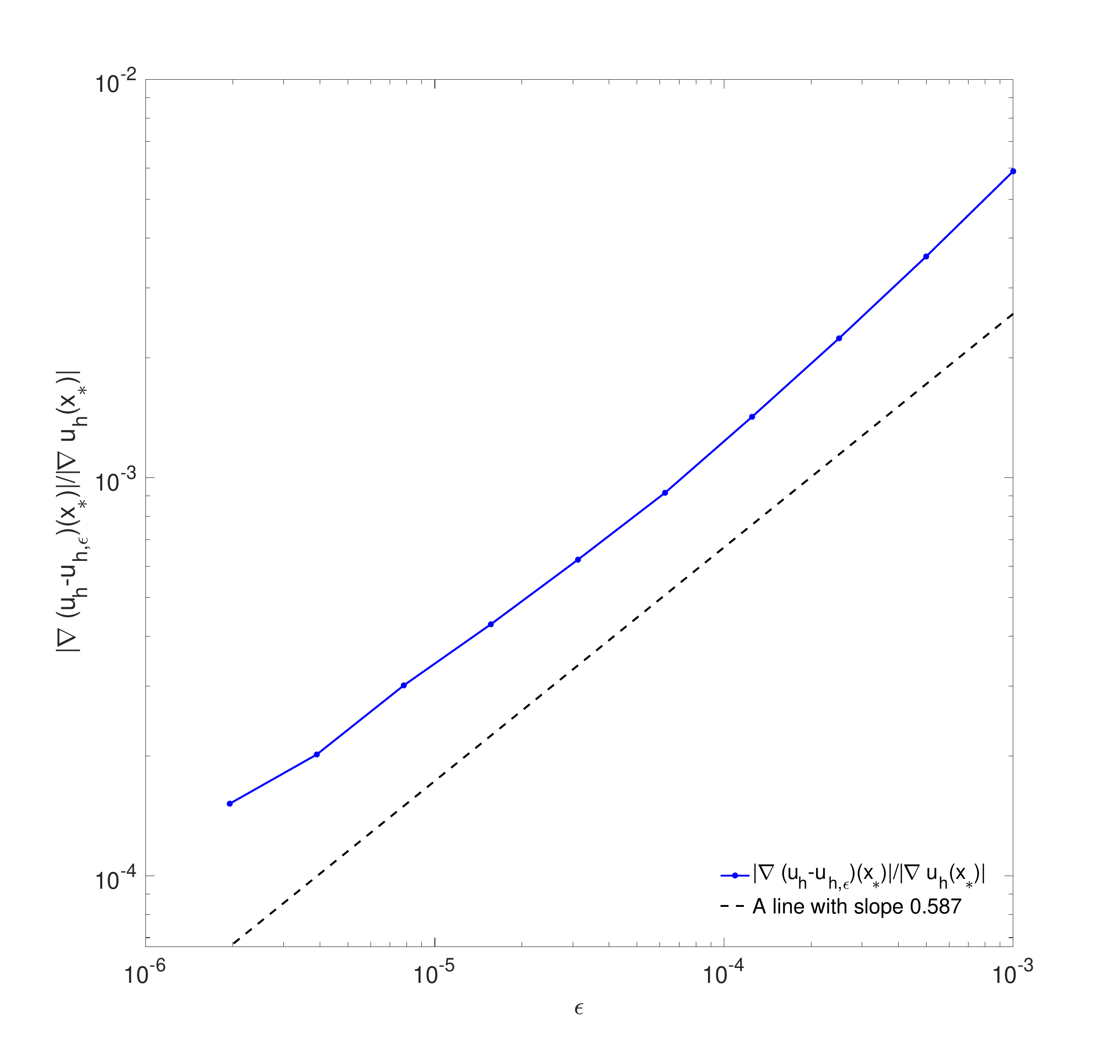}}
\caption{Example 3: The relative energy error (left) and the relative error at $x_*$ (right)} \label{err_H1_interiori}
\end{figure}

\section{Appendix}\label{sec_geo}

In this appendix, we study the influence of the geometric modeling error in the setting of Example \ref{example3}. Let $D_1,D_2\subset\R^2$ be two Lipschitz domains so that $D_1\subset D_2$, see Fig.\ref{fig:5.1}. Set $G_1=D_2\backslash\bar D_1$. Let $u_i\in H^1(D_i)$, $i=1,2$, be the solution of the problem
\beq\label{s1}
-\Delta u_i=f_i\ \ \mbox{in }D_i,\ \ \ \ u_i=0\ \ \mbox{on }\pa D_i,
\eeq
where $f_i\in L^2(D_i)$ satisfying $f_2=f_1$ in $D_1$. By Jerison and Kenig \cite[Theorem 0.5]{Jerison}, there exists $p_i>2$ such that $u_i\in W^{1,p_i}(D_i)$ and
\beq\label{s2}
\|u_i\|_{W^{1,p_i}(D_i)}\le C(p_i,D_i)\|f_i\|_{W^{-1,p_i}(D_i)}
\eeq
for some constant $C(p_i,D_i)$ depending on $p_i,D_i$. Here $W^{-1,p_i}(D_i)$ is the dual space of $W^{1,p_i'}_0(D_i)$, where $1/p_i+1/p_i'=1$.

\begin{figure}[ht!]
\centering
\includegraphics[width=0.3\textwidth]{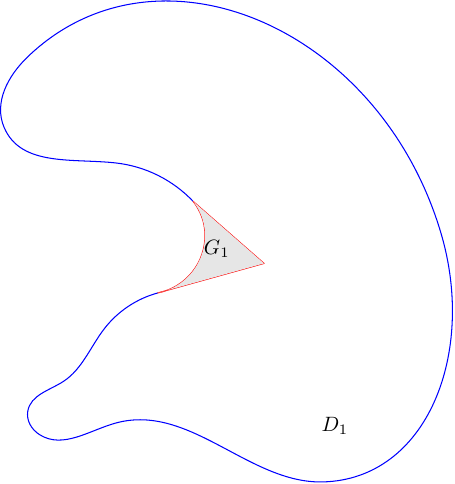}
\caption{Illustration of domains of the problems \eqref{s1}.}\label{fig:5.1}
\end{figure}

Let $B\subset\R^2$ be some circle including $D_2$ and let $\widetilde f_2=f_2\chi_{D_2}$, where $\chi_{D_2}$ is the characteristic function of $D_2$. Let $\widetilde u_1\in H^1(B\backslash\bar D_1)$ be the extension of $u_1$ with continuing flux such that
\beq\label{s3}
-\Delta \widetilde u_1=\widetilde f_2\ \ \mbox{in }B\backslash\bar D_1,\ \ \frac{\pa\widetilde u_1}{\pa n_1}=\frac{\pa u_1}{\pa n_1}\ \ \mbox{on }\pa D_1,\ \ \widetilde u_1=0\ \ \mbox{on }\pa B,
\eeq
where $n_1$ is the unit normal to $\pa D_1$. Since $u_1\in W^{1,p_1}(D_1)$, it is easy to see that $\frac{\pa u_1}{\pa n_1}\in W^{-1/p_1,p_1}(\pa D_1)$ and $\|\pa u_1/\pa n_1\|_{W^{-1/p_1,p_1}(\pa D_1)}\le C\|u_1\|_{W^{1,p_1}(D_1)}$, where $W^{-1/p_1,p_1}(\pa D_1)$ is the dual space of $W^{1/p_1,p_1'}(\pa D_1)$. By Geng \cite[Theorem 1.2]{Geng}, if $f_2\in L^{p_2}(D_2)$, we know that there exists a $q\in (2,\min(p_1,p_2))$, such that
\be
\|\widetilde u_1\|_{W^{1,q}(B\backslash\bar D_1)}
&\le&C(q,D_1)\left(\|\widetilde f_2\|_{L^{q}(B\backslash\bar D_1)}+\Big\|\frac{\pa u_1}{\pa n_1}\Big\|_{W^{-1/q,q}(\pa D_1)}\right)\nn\\
&\le&C(q,D_1)(\|f_2\|_{L^{q}(D_2)}+\|u_1\|_{W^{1,q}(D_1)}),\label{s4}
\ee
where the constant $C(q,D_1)$ depends on $q$ and $B\backslash\bar D_1$.
The following theorem is the main result of this appendix.

\begin{thm}\label{thm:5.1}
Let $\epsilon=|G_1|$ and $\sigma=(q-2)/(2q)$, where $q\in (2,min(p_1,p_2))$. If $f_2\in L^{p_2}(D_2)$ and $u_i, i=1,2$, is the solution of \eqref{s1}, there exists a constant $C$ independent of $u_i,f_i$, $i=1,2$, such that

{\rm (i)} $\|\na(u_1-u_2)\|_{L^2(D_1)}\le C\epsilon^\sigma(\|f_1\|_{W^{-1,p_1}(D_1)}+\|f_2\|_{L^{p_2}(D_2)})$;

{\rm (ii)} $|\na(u_1-u_2)(x_*)|\le Cd^{-2}\epsilon^{2\sigma}(\|f_1\|_{W^{-1,p_1}(D_1)}+\|f_2\|_{L^{p_2}(D_2)})$, where $x_*\in D_1$ with $B(x_*,d)\subset D_1$.
\end{thm}


\begin{proof}
We first show (i). By integrating by parts, since $u_i=0$ on $\pa D_i$, $i=1,2$, and $\Delta (u_1-u_2)=0$ in $D_1$, by \eqref{s3} we have
\beq\label{s11}
\|\na(u_1-u_2)\|_{L^2(D_1)}^2=-\int_{\pa D_1}\frac{\pa (u_1-u_2)}{\pa n_1}u_2=\int_{\pa G_1}\frac{\pa(\widetilde u_1-u_2)}{\pa n}u_2,
\eeq
where $n$ is the unit outer normal to $\pa G_1$. By H\"older inequality, for any $q>2$, $\|v\|_{L^2(G_1)}\le\epsilon^\sigma\|v\|_{L^q(G_1)}\ \ \forall v\in L^q(G_1)$. Now since $\Delta (\widetilde u_1-u_2)=0$ in $G_1$, it follows from \eqref{s11}, \eqref{s1} and \eqref{s4} that
\begin{align}
\|\na(u_1-u_2)\|_{L^2(D_1)}^2=\int_{G_1}\na(\widetilde u_1-u_2)\cdot\na u_2&\le&\epsilon^{2\sigma}\|\na(\widetilde u_1-u_2)\|_{L^q(G_1)}\|\na u_2\|_{L^q(G_1)}\nn\\
&\le&C\epsilon^{2\sigma}(\|f_1\|_{W^{-1,q}(D_1)}+\|f_2\|_{L^{q}(D_2)})^2.\label{ss}
\end{align}
This shows (i) as $q\le \min(p_1,p_2)$. To prove (ii), we use the interior gradient estimate for harmonic functions (see e.g., Axler et al. \cite[Corollary 8.2]{Axler2013}) to obtain $|\na(u_1-u_2)(x_*)|\le Cd^{-2}\|u_1-u_2\|_{L^2(D_1)}$. Now we estimate $\|u_1-u_2\|_{L^2(D_1)}$ by the duality argument. First we note that
\beq\label{s5}
\int_{D_1}\na (u_1-u_2)\cdot\na v=0\ \ \forall v\in H^1_0(D_1).
\eeq
Let $w_j\in H^1_0(D_j)$, $j=1,2$, be the solution of the problems
\beq\label{s6}
-\Delta w_1=u_1-u_2\ \ \mbox{in }D_1,\ \ \ \ -\Delta w_2=(u_1-u_2)\chi_{D_1}\ \ \mbox{in }D_2,
\eeq
where $\chi_{D_1}$ is the characteristic function of $D_1$. By \eqref{s1} we have for $i=1,2$,
\beq\label{s12}
\|w_i\|_{W^{1,q}(D_i)}\le C\|w_i\|_{W^{1,p_i}(D_i)}\le C\|u_1-u_2\|_{W^{-1,p_i}(D_1)}\le C\|u_1-u_2\|_{L^2(D_1)}.
\eeq
By \eqref{ss} and \eqref{s2} we have
\be\label{s10}
\|\na(w_1-w_2)\|_{L^2(D_1)}&\le&C\epsilon^\sigma(\|u_1-u_2\|_{W^{-1,q}(D_1)}+\|u_1-u_2\|_{L^q(D_1)})\nn\\
&\le&C\epsilon^\sigma(\|f_1\|_{W^{-1,p_1}(D_1)}+\|f_2\|_{L^{p_2}(D_2)}).
\ee
Moreover, by \eqref{s1}, $\|w_2\|_{W^{1,p_2}(D_2)}\le C\|u_1-u_2\|_{L^2(D_1)}$. By the second equation in \eqref{s6}
\be\label{s9}
\|u_1-u_2\|_{L^2(D_1)}^2=-\int_{D_1}\Delta w_2(u_1-u_2)=\int_{D_1}\na w_2\cdot\na (u_1-u_2)-\int_{\pa D_1}\frac{\pa w_2}{\pa n_1}u_2.
\ee
By \eqref{s6}, \eqref{s10} and (i) we have
\be\label{s7}
\int_{D_1}\na w_2\cdot\na (u_1-u_2)&\le&\|\na( w_1-w_2)\|_{L^2(D_1)}\|\na(u_1-u_2)\|_{L^2(D_1)}\nn\\
&\le&C\epsilon^{2\sigma}(\|f_1\|_{W^{-1,p_1}(D_1)}+\|f_2\|_{L^{p_2}(D_2)})^2.
\ee
Again, since $u_2=0$ on $\pa D_2$, $\Delta w_2=0$ in $G_1$, by \eqref{s12} we obtain
\be
-\int_{\pa D_1}\frac{\pa w_2}{\pa n_1}u_2=\int_{\pa G_1}\frac{\pa w_2}{\pa n}u_2&=&\int_{G_1}\na u_2\cdot\na w_2\nn\\
&\le&C\epsilon^{2\sigma}\|u_2\|_{W^{1,q}(G_1)}\|w_2\|_{W^{1,q}(G_1)}\nn\\
&\le&C\epsilon^{2\sigma}\|f_2\|_{W^{-1,p_2}(D_2)}\|u_1-u_2\|_{L^2(D_1)}.\label{s8}
\ee
Inserting \eqref{s7}-\eqref{s8} into \eqref{s9} we obtain
\ben
\|u_1-u_2\|_{L^2(D_1)}\le C\epsilon^{2\sigma}(\|f_1\|_{W^{-1,p_1}(D_1)}+\|f_2\|_{L^{p_2}(D_2)}).
\een
This completes the proof.
\end{proof}

\end{document}